\documentclass[11pt,a4paper]{article}
\usepackage{a4wide,amsfonts,amsmath,latexsym,amsthm,amssymb,euscript,graphicx,units,mathrsfs, color}
\usepackage{diagbox}
\usepackage{geometry}
\usepackage{float}
\usepackage{amssymb}
\newfloat{figure}{H}{lof}
\floatname{figure}{\figurename}

\usepackage[T1]{fontenc}

\usepackage[francais,english]{babel}
\usepackage[applemac]{inputenc}

\newtheorem{Theorem}{Theorem}[section]
\newtheorem{Proposition}{Proposition}[section]

\newtheorem{Lemma}{Lemma}[section]

\newtheorem{Remark}{Remark}[section]

\numberwithin{equation}{section}

\def\E{\mathbb{E}}
\def\P{\mathbb{P}}

\def\real{\mathbb{R}}

\def\S{\mathbb{S}}
\def\H{\mathbb{H}}

\def\1{\textbf{1}}

\def\F{\mathbb{F}}

\title{Density analysis of non-Markovian BSDEs and applications to biology and finance.}
\author{Thibaut Mastrolia\footnote{Universit\'e Paris-Dauphine, CEREMADE UMR CNRS 7534, Place du Mar\'echal De Lattre De Tassigny, 75775 Paris Cedex 16, FRANCE, \texttt{mastrolia@ceremade.dauphine.fr}}}
\begin{document}
\maketitle
\centering \Large 
\vspace{2em}

\abstract{\noindent In this paper, we provide conditions which
ensure that stochastic Lipschitz BSDEs admit Malliavin differentiable solutions. We investigate the problem of existence of densities for the first components
of solutions to general path-dependent stochastic Lipschitz BSDEs and obtain results for the second components in particular cases. We apply these results to both the study of a gene
expression model in biology and to the classical pricing problems in mathematical finance.
\vspace{1em}

\vspace{1em}

{\noindent \textit{AMS 2010 subject classification:} Primary: 60H10; Secondary: 60H07, 91G30, 92D20.~\newline
\normalsize
}

\vspace{1em}
\noindent

\section{Introduction}
The problem of existence of densities for random processes, as \textit{e.g.} solutions of stochastic differential equations (SDEs), has been a very active strand of research in the last two decades, see among others \cite{KohatsuH_EllipticRandomVariable, Nualart_Book}. A very useful criterion to prove that the law of a random variable admits a density is the criterion of Bouleau and Hirsch, see \textit{e.g.} \cite[Theorem 2.1.2]{Nualart_Book}. The analysis of densities has been the subject of several works dealing with Stochastic Partial
Differential Equations (SPDEs), among which we can mention the study of the stochastic heat equation, the stochastic wave equation (see for instance \cite{MuellerNualart}, \cite{NualartQuerS}, \cite{MilletS-S}), the Navier-Stokes equation \cite{DebusscheRomito} and recently the Landau equation for Maxwellian molecules (see \cite{DelarueNualartMenozzi}). Besides, most of these papers investigate tails estimates of the solutions to SPDEs by using the formula of Nourdin and Viens, introduced in \cite{NourdinViens}, to have a better understanding of these processes.\vspace{0.3em}

\noindent Although the problem of existence of densities for S(P)DEs, together with estimates on their tails, has been a prosperous field, the corresponding theory for Backward Stochastic Differential Equations (BSDEs) has not received the same attention in the literature. BSDEs were introduced for the first time in 1973 by Bismut in \cite{Bismut}, in order to study stochastic control problems and their links to the Pontryagin maximum principle. The theory of BSDEs was then formalised and developed in the 90's, with the seminal papers \cite{PardouxPeng_AdaptedSolutionBSDE,PardouxPeng_BSDEQuasilinearParabolicPDE} and \cite{ELKPengQuenez}. In the last decades, BSDEs have been the object of an ever growing interest, since these equations naturally appear in financial problems, as for instance pricing problems (see \cite{ELKPengQuenez}) and utility maximisation problems (see \cite{ELKRouge}, \cite{HuImkellerMuller}).\vspace{0.3em}

\noindent As far as we know, the existence of densities for solutions to BSDEs was studied in three papers. Conditions ensuring that the first component $Y$ of the solution to a Lipschitz BSDE admits a density were provided for the first time in \cite{AntonelliKohatsu}. In this paper, the authors also investigated both estimates on the existing density and its smoothness. Then, a result ensuring existence of a density for the second component $Z$ of the solution to a particular BSDE, in which the generator is linear with respect to its $z$ variable, was obtained in \cite{AbouraBourguin}. Recently, this problem was studied in \cite{MastroliaPossamaiReveillac_Density} for both the $Y$ and the $Z$ components of solutions to BSDEs with a quadratic growth generator. However, \cite{AntonelliKohatsu, MastroliaPossamaiReveillac_Density} only consider Markovian BSDEs, that is the case where the data $\xi$ and $\omega\longmapsto f(s,\omega,y,z)$ of such equations are only random through a Markovian process, and \cite{AbouraBourguin} only considers the semi-Markovian case, that is the case where only $\omega\longmapsto f(s,\omega,y,z)$ is Markovian.\vspace{0.3em}

\noindent Although the previous studies are interesting from a mathematical point of view, these results seem to be too restrictive for applications. As an example, consider a pricing problem which could be reduced to solve the following BSDE (see \cite{ELKPengQuenez} for more details)
$$dY_t=(r_tY_t+\theta_t Z_t )dt+Z_t dW_t,\, Y_T=\xi ,$$
where $r$ denotes the interest rate of the market, $\theta$ is the market price of risk and $\xi$ is the liability. As noticed in \cite{ELKHuang}, assuming that $r$ is bounded, for instance, is not realistic. This remark led the authors of \cite{ELKHuang} to define a new class of BSDEs satisfying a so-called stochastic Lipschitz condition for their generator. Existence and uniqueness results have been obtained for this class of BSDEs first in \cite{ELKHuang}, and have then been extended in \cite{BenderKohlmann, wang_ran_chen, BriandConfortola} among others.\vspace{0.5em}

\noindent The problem of existence of densities for the laws of components of solutions to stochastic Lipschitz BSDEs has not been studied yet, and \textit{a fortiori} in the non-Markovian framework, \textit{i.e.} when neither the terminal condition $\xi$ nor $\omega\longmapsto f(s,\omega,y,z)$ depend on the randomness through a Markovian process. We give in the present paper conditions on $\xi$ and $f$ to solve these problems. Besides, although it is well-known that under suitable conditions on the data, a  non-Markovian stochastic Lipschitz BSDEs admits a unique solution (see \cite{ELKHuang, wang_ran_chen, BenderKohlmann, BriandConfortola}), the Malliavin differentiability of the solutions to such BSDEs has not been studied yet in the general case. In order to apply Bouleau and Hirsch's Criterion (\cite[Theorem 2.1.2]{Nualart_Book}) to solve the problem of existence of densities for the law of $Y$ and $Z$, we provide also in this paper conditions which ensure that the components $Y$ and $Z$ solutions to non-Markovian stochastic Lipschitz BSDEs are Malliavin differentiable.\vspace{0.3em}

\noindent The structure of this paper is the following. After some preliminaries and notations in Section \ref{bio_preliminaries}, we provide in Section \ref{section:MDslipsch} two approaches to study the Malliavin differentiability of solutions to stochastic Lipschitz BSDEs. Indeed, in view of the classical literature, we distinguish two types of assumptions which provide existence and uniqueness of solutions to stochastic Lipschitz BSDE. On the one hand, we have assumptions as in \cite{ELKHuang,BenderKohlmann, wang_ran_chen} dealing with $\beta$-spaces (see $\mathbb S_{2p,\beta}$ and $\mathbb H_{2p,\beta}$ below), on the other hand, we have assumptions dealing (mainly) with the BMO-norm of the data, as in \cite{BriandConfortola}. We then reach in this paper two kind of conditions which ensure that the components of the solution $(Y,Z)$ to a stochastic Lipschitz BSDE are Malliavin differentiable. The first one, investigated in Section \ref{sto:lip:1ereapproche}, is based on the papers \cite{ELKHuang,BenderKohlmann, wang_ran_chen}. Using \textit{{\it a priori}} estimates for solutions to stochastic Lipschitz BSDE, obtained in \cite[Proposition 3.6]{wang_ran_chen}, we have conditions on the data of such BSDE which provide the Malliavin differentiability of $Y$ and $Z$ (see Assumption $\mathbf{(DsL^{p,\beta}})$). The second approach, studied in Section \ref{sto:lip:2emeapproche}, is based on the papers \cite{BriandConfortola, AIR}. We give assumptions, similar to those obtained in \cite{MastroliaPossamaiReveillac_MDBSDE}, see Assumptions ($\mathbf{sH_{1,\infty}}$) and ($\mathbf{sH_{2,\infty}}$) below, which ensure that $Y$ and $Z$ are Malliavin differentiable. We then compare these two approaches, and the corresponding conditions, in Section \ref{Mbio:comparaison}.\vspace{0.5em}

\noindent By taking advantage of the results obtained in Section \ref{section:MDslipsch}, we deal in Section \ref{Mbio:densityY} with the problem of existence of densities for the laws of solutions to stochastic Lipschitz BSDE in the non-Markovian case. We give in Section \ref{chapBioMbio:section:density:stolip} conditions which ensure the existence of densities for the law of the $Y$ component of the solution to stochastic Lipschitz BSDEs, by using Bouleau and Hirsch's Criterion. We provide weaker conditions in Section \ref{chapBioMbio:section:dnesity:lip} for the $Y$ component of the solution to a non-Markovian Lipschitz BSDE. We then turn to the $Z$ component in Section \ref{chapBioMbio:section:discussion}. We first provide in Section \ref{soussection:z:linear} conditions ensuring that the law of the $Z_t $ component has a density for a particular class of BSDE, extending the results of \cite{AbouraBourguin}. We then explain in Section \ref{someremark:z} why we are not able to adapt the proofs of \cite{MastroliaPossamaiReveillac_Density} to the non-Markovian framework for the $Z$ components of solutions to general non-Markovian BSDEs and we indicate paths for future researches. We finally apply our study in Sections \ref{chapBio:section:gene} and \ref{Mbio:section:finance} to biology and finance respectively.\vspace{0.5em}

\noindent In Section \ref{chapBio:section:gene}, we propose to study mathematically a model of synthesis of proteins introduced in \cite{shamarova_etal.}, with the Malliavin calculus. Indeed, in order to validate their model, the authors of \cite{shamarova_etal.} need to compare the law of the protein concentration at time $t$ obtained by solving a BSDE with the data produced by Gillespie Method (see \cite{Gillespie77}). However, in \cite{shamarova_etal.}, the authors assumed implicitly that the law of the first component $Y_t$ of the BSDE under consideration admits a density with respect to the Lebesgue measure. The present paper can be seen as a mathematical strengthening of the model developed in \cite{shamarova_etal.} by using the so-called Nourdin and Viens' formula to obtain Gaussian estimates of the density. Besides, we propose to extend their model to the non-Markovian setting, which could be quite relevant when we study the synthesis of protein in some models (see for instance \cite{BratsunVolfsonHastyTsimring,Leoncini_Thesis,FromionLeonciniRobert}). \vspace{0.5em}

\noindent In Section \ref{Mbio:section:finance}, we study classical pricing problems. As showed in \cite{ELKPengQuenez}, this problem can be reduced to solve a stochastic linear BSDE. In this section we aim at applying the results obtained in previous sections to Asian and Lookback options in the Va\v{s}\`i\v{c}ek Model to obtain information on both the regularity of the value function and the regularity of optimal strategies.
\section{Preliminaries and notations}\label{bio_preliminaries}
\subsection{Notations} We denote by $\lambda$ the Lebesgue measure on $\mathbb R$. Let $T>0$ be a time fixed horizon. Let $\Omega:=C_0([0,T],\mathbb R)$ be the canonical Wiener space of continuous function $\omega$ from $[0,T]$ to $\mathbb R$ such that $\omega(0)=0$. We denote by $W:=(W_t)_{t\in [0,T]}$ the canonical Wiener process, that is, for any time $t$ in $[0,T]$, $W_t(\omega):=\omega_t$ for any element $\omega$ in $\Omega$. We set $\mathbb F^o$ the natural filtration of $W$. Under the Wiener measure $\mathbb P$, the process $W$ is a standard Brownian motion and we denote by $\mathbb F:=(\mathcal F_t)_{t\in[0,T]}$ the usual right-continuous and complete augmentation of $\mathbb F^o$ under $\mathbb P$. For the sake of simplicity, we denote all expectations under $\mathbb P$ by $\mathbb E$ and we set for any $t\in [0,T]$ $\mathbb E_t[\cdot]:=\mathbb E [\cdot | \mathcal F_t]$. Besides, all notions of measurability for elements of $\Omega$ will be with respect to the filtration $\mathbb F$ or the $\sigma$-field $\mathcal F_T$. \vspace{0.5em}

\noindent We set $\mathfrak h:=L^2([0,T],\mathbb R)$, where $\mathcal B([0,T])$ is the Borel $\sigma$-algebra on $[0,T]$, and consider the following inner product on $\mathfrak h$
$$\langle f,g\rangle :=\int_0^Tf(t)g(t)dt, \quad \forall (f,g) \in \mathfrak h^2,$$
with associated norm $\|{\cdot}\|_{\mathfrak h}$. 
Let now $H$ be the Cameron-Martin space that is the space of functions in $\Omega$ which are absolutely continuous with square-integrable derivative and which start from $0$ at $0$:
$$H:=\left\{ h:[0,T] \longrightarrow \real, \; \exists \dot{h}\in\mathfrak h, \; h(t)=\int_0^t \dot{h}(x)dx, \; \forall t\in [0,T]\right\}.$$
For any $h$ in $H$, we will always denote by $\dot{h}$ a version of its Radon-Nykodym density with respect to the Lebesgue measure. 
Notice that $H$ is an Hilbert space equipped with the inner product $\langle h_1,h_2 \rangle_{H}:=\langle \dot{h}_1,\dot{h}_2 \rangle_{\mathfrak h}$, for any $(h_1, h_2)\in H\times H$, and with associated norm $\|h\|_H^2:=\langle \dot h, \dot h \rangle_\mathfrak h$.
Let $p\geq 1$. Define  $L^p(\mathcal K)$ as the set of all $\mathcal F_T$-measurable random variables $F$ which are valued in an Hilbert space $\mathcal{K}$, and such that $\|F\|_{L^p(\mathcal K)}<+\infty$, where
$$\|F\|_ {L^p(\mathcal K)}:=\E\left[\|F\|_ {\mathcal K}^p\right]^{1/p},$$
where the norm $\| \cdot \|_{\mathcal K}$ is the one canonically induced by the inner product on $\mathcal K$. 
We define
$$L^p([t,T];\mathcal K):=\left\{f:[t,T]\longrightarrow\mathcal K, \text{ Borel-measurable, s.t. }\int_t^T \| f(s)\|^p_{\mathcal{K}}ds <+\infty \right\}.$$

\noindent Set $\rm{BMO}(\mathbb P)$ as the space of square integrable, continuous, $\mathbb R$-valued martingales $M$ such that
$$\|M\|_{\rm{BMO}}:=\underset{\tau\in\mathcal T_0^T}{\text{esssup}}\left\|\mathbb E_\tau\left[\left(M_T-M_{\tau}\right)^2\right]\right\|_{\infty}<+\infty,$$
where for any $t\in[0,T]$, $\mathcal T_t^T$ is the set of $\F$-stopping times taking their values in $[t,T]$. Accordingly, $\mathbb H^2_{\rm{BMO}}$ is the space of $\mathbb R$-valued and $\F$-predictable processes $Z$ such that
$$\|Z\|^2_{\mathbb H^2_{\rm{BMO}}}:=\left\|\int_0^.Z_sdW_s\right\|_{\rm{BMO}}<+\infty.$$
Denoting by $\mathcal E(M)$ the stochastic exponential of a semi-martingale $M$, we have finally the following result 

\begin{Theorem}\cite[Theorem 2.3]{Kazamaki}\label{kazamakithm}
If $M\in \mathrm{BMO}(\mathbb P)$ then $\mathcal E\left( M\right)$ is a uniformly integrable martingale.
\end{Theorem}
\noindent For any nonnegative $\mathbb F$-adapted process $\alpha$, we define the following increasing and continuous process
$$A_t^\alpha:= \int_0^t \alpha_s^2 ds.$$
Let $p>\frac12$, $\beta >0$ and let $\alpha$ be a nonnegative $\mathbb F$-adapted process, we define
\begin{align*}
\mathbb S_{2p,\beta, \alpha}&:= \left\{Y \text{ adapted and c\`adl\`ag, } \|Y \|_{\mathbb S_{2p,\beta, \alpha}}^{2p}:= \mathbb E\left[ \underset{0\leq t\leq T}{\sup} e^{p\beta A^\alpha_t} |Y_t|^{2p}\right]<+\infty \right\}.\\
\mathbb H_{2p,\beta, \alpha}&:= \left\{Y \text{ progressively measurable, } \|Y \|_{\mathbb H_{2p,\beta, \alpha}}^{2p}:= \mathbb E\left[\left(\int_0^T e^{\beta A^\alpha_s} |Y_s|^2 ds\right)^p\right]<+\infty \right\}. \\
\mathbb H^a_{2p,\beta, \alpha}&:=  \left\{Y \text{ progressively measurable, } \|Y \|_{\mathbb H^\alpha_{2p,\beta, \alpha}}^{2p}:= \mathbb E\left[\int_0^T \alpha_s^2e^{\beta A^\alpha_s} |Y_s|^{2p} ds\right]<+\infty \right\}.
\end{align*}

\noindent To match with the notations in \cite{BriandConfortola}, we define for any real $p>0$ the spaces $\mathcal S_p$ and $\mathcal H_p$ by
\begin{align*}
\mathcal S_p&:= \left\{Y \text{ adapted and c\`adl\`ag processes, } \|Y \|_{\mathcal S_{p}}:= \mathbb E\left[ \underset{0\leq t\leq T}{\sup} |Y_t|^{p}\right]^{1\wedge 1/p}<+\infty \right\}\\
\mathcal H_p&:= \left\{Z \text{ predictable processes, } \|Z \|_{\mathcal H_{p}}:= \mathbb E\left[\left(\int_0^T |Z_s|^2 ds\right)^{p/2}\right]^{1\wedge 1/p}<+\infty \right\}.
\end{align*}

\noindent  In particular, for any $p>\frac12$ we have $\mathcal S_{2p}=\mathbb S_{2p, 0, \alpha}$ and $\mathcal H_{2p}=\mathbb H_{2p, 0, \alpha}$. Notice moreover that the following inequality holds for any $p>\frac12, \; \beta>0$ and for any nonnegative $\mathbb F$-adapted process $\alpha$
\begin{equation}
\label{inegalite:norme}\|Y\|_{\mathcal S^{2p}}^{2p}+\| Z\|_{\mathcal H^{2p}}^{2p}\leq \| Y\|_{\mathbb S_{2p,\beta, \alpha}}^{2p}+\| Z\|_{\mathbb H_{2p,\beta, \alpha}}^{2p}.
\end{equation}

\subsection{Elements of Malliavin calculus}

We give in this section some results on the Malliavin calculus that we will use in this paper. Let now $\mathcal{S}$ be the set of cylindrical functionals, that is the set of random variables $F$ of the form
\begin{equation}
\label{eq:cylindrical}
F=f(W(h_1),\ldots,W(h_n)), \quad (h_1,\ldots,h_n) \in H^n, \; f \in C^\infty_b(\real^n), \text{ for some }n\geq 1,
\end{equation}
where $W(h):=\int_0^T \dot{h}_s dW_s$ for any $h$ in $H$ and where $C^\infty_b(\real^n)$ denotes the space of bounded mappings which are infinitely continuously differentiable with bounded derivatives. For any $F$ in $\mathcal S$ of the form \eqref{eq:cylindrical}, the Malliavin derivative $\nabla F$ of $F$ is defined as the following $H$-valued random variable:
\begin{equation}
\label{eq:DF}
\nabla F:=\sum_{i=1}^n f_{x_i}(W(h_1),\ldots,W(h_n)) h_i,
\end{equation}
where $f_{x_i}:=\frac{df}{dx_i}$.
It is then customary to identify $\nabla F$ with the stochastic process $(\nabla_t F)_{t\in [0,T]}$. Denote then by $\mathbb{D}^{1,p}$ the closure of $\mathcal{S}$ with respect to the Malliavin-Sobolev semi-norm $\|\cdot\|_{1,p}$, defined as:
$$ \|F\|_{1,p}:=\left(\E\left[|F|^p\right] + \E\left[\|\nabla F\|_{H}^p\right]\right)^{1/p}. $$ We set $\mathbb D^{1,\infty}:= \bigcap_{p\geq 2} \mathbb D^{1,p}$.
We make use of the notation $DF$ to represent the derivative of $\nabla F$ as: 
$$ \nabla_t F=\int_0^t D_s F ds, \quad t\in [0,T]. $$ 
\vspace{0.5em}

\noindent  To avoid any ambiguity in the non-Markovian case we will consider later on, we need to introduce immediately some further notations. For any mapping $\tilde f$ from $[0,T]\times \Omega\times \mathbb R$ into $\mathbb R$,  we let $D\tilde f(t,y)$ be the Malliavin derivative, computed at the point $(t,y)$, of $\omega\longmapsto \tilde f(t,\omega,y)$. If $D\tilde f$ is continuously differentiable with respect to $y$, we denote by $(D\tilde f)_y$ its derivative with respect to $y$. Let now $Y$ be an $\mathbb F$-progressively measurable real process, with $Y_t\in\mathbb D^{1,2}$ at time $t\in [0,T]$. Using the chain rule formula (see for instance \cite{Nualart_Book}), the Malliavin derivative of $D\tilde f$ at $(t,Y_t)$, denoted by $D^2 \tilde f(t,Y_t)$ is given by
\begin{equation}\label{D2:definition:symetrie}D^2_{v,u}\tilde f(t,Y_t):= D_v(D_u \tilde f)(t,Y_t) + (D_u \tilde f)_y(t,Y_t) D_v Y_t,\; 0\leq u,v\leq t.  \end{equation}

\noindent Let $h$ be in $H$ and let $\tau$ be the following shift operator $\tau_{h}:\Omega\longrightarrow\Omega$ defined by
$$\tau_{h}(\omega):=\omega + h, \; \omega\in \Omega.$$
Note that the fact that $h$ belongs to $H$ ensures that $\tau_h$ is a measurable shift on the Wiener space. In the present paper, we will use the characterization of the Malliavin differentiability, as a convergence of a difference quotient in $L^p$, introduced in \cite{MastroliaPossamaiReveillac_MDBSDE}, recalled below. 

\begin{Theorem}[Theorem 4.1 in \cite{MastroliaPossamaiReveillac_MDBSDE}]\label{thm:newcarD12}
Let $p>1$ and $F\in L^p(\real)$. Then
$F$ belongs to $\mathbb D^{1,p}$ if and only if there exists $\mathcal{D}F$ in $L^p(H)$ and there exists $q\in [1,p)$ such that for any $h$ in $H$
$$ \lim\limits_{\varepsilon\to 0} \E\left[ \left| \frac{F\circ \tau_{\varepsilon h}-F}{\varepsilon}-\langle \mathcal{D}F, h\rangle_H\right|^q\right]=0. $$
In that case, $\mathcal D F= \nabla F$.
\end{Theorem}
\noindent We now recall the criterion that we will use to check the absolute continuity of the law of a random variable $F$ with respect to the Lebesgue measure.

\begin{Theorem}[Bouleau-Hirsch Criterion, see e.g. Theorem 2.1.2 in \cite{Nualart_Book}]\label{BH}
Let $F$ be in $\mathbb{D}^{1,p}$ for some $p>1$. Assume that $\|DF\|_{\mathfrak h} >0$, $\mathbb{P}-$a.s. Then $F$ has a probability distribution which is absolutely continuous with respect to the Lebesgue measure on $\mathbb{R}$.
\end{Theorem}

\noindent Let $F$ such that $\|DF\|_{\mathfrak h} >0$, $\mathbb{P}-$a.s., then the previous criterion implies that the law of $F$ admits a density $\rho_F$ with respect to the Lebesgue measure. Assume that there exists in addition a measurable mapping $\Phi_F$ with $\Phi_F : \mathbb{R}^{\mathfrak h}  \rightarrow \mathfrak h$, such that $DF=\Phi_F(W)$, we then set:
\begin{equation}\label{def:gF}
g_F(x):=\int_0^\infty e^{-u} \mathbb{E}\left[\mathbb{E}^*[\langle \Phi_F(W),\widetilde{\Phi_F^u}(W)\rangle_{\mathfrak h}] \Big{|} F-\mathbb{E}(F)=x\right] du, \ x\in \real,
\end{equation} 
where $\widetilde{\Phi_F^u}(W):=\Phi_F(e^{-u}W+\sqrt{1-e^{-2u}}W^*)$ with $W^*$ an independent copy of $W$ defined on a probability space $(\Omega^*,\mathcal{F}^*,\mathbb{P}^*)$, and $\mathbb{E}^*$ denotes the expectation under $\mathbb{P}^*$ ($\Phi_F$ being extended on $\Omega\times \Omega^*$). We recall the following result from \cite{NourdinViens}.

\begin{Theorem}[Nourdin-Viens' Formula]\label{thm_NourdinViens}
The law of a random variable $F$ has a density $\rho_F$ with the respect to the Lebesgue measure if and only if the random variable $g_F(F-\mathbb{E}[F])$ is positive a.s. In this case, the support of $\rho_F$, denoted by $\text{supp}(\rho_F)$, is a closed interval of $\mathbb{R}$ and for all $x \in \text{supp}(\rho_F)$:
\begin{equation*}
\rho_F(x)=\frac{\mathbb{E}(|F-\mathbb{E}[F]|)}{2g_{F}(x-\E[F])}\exp{\left( -\int_0^{x-\mathbb{E}[F]} \frac{udu}{g_F(u)} \right)}.
\end{equation*}
\end{Theorem}

\section{Malliavin differentiability of stochastic Lipschitz BSDEs}\label{section:MDslipsch}
The Malliavin differentiability of solutions to non-Markovian Lipschitz BSDE has been studied first in \cite{ELKPengQuenez} and more recently in \cite{MastroliaPossamaiReveillac_MDBSDE}, as well as in \cite{GeissSteinicke} for L\'evy driven BSDE. In \cite{GeissSteinicke}, the authors use the well-known characterization of the Malliavin derivative as G\^ateaux derivative introduced by Sugita in \cite{sugita} and they obtain similar conditions, for the Brownian part, to those in \cite{ELKPengQuenez} (see \cite[Section 4, $\mathbf{(A_f)}$]{GeissSteinicke}), while \cite{MastroliaPossamaiReveillac_MDBSDE} took the advantage of a new $L^p$ characterization of the Malliavin differentiability (see Theorem \ref{thm:newcarD12}) to improve conditions obtained in \cite{ELKPengQuenez}.\vspace{0.3em}

\noindent Here, we extend the results of \cite{MastroliaPossamaiReveillac_MDBSDE} to the stochastic Lipschitz case. We consider the following non-Markovian BSDE
\begin{equation}\label{chapBioedsrslips}
Y_t=\xi+\int_t^T f(s,Y_s,Z_s) ds -\int_t^T Z_s dW_s, \; \forall t\in [0,T], \; \mathbb P-a.s.
\end{equation}
 where $\xi$ is an $\mathcal{F}_T$-measurable random variable and $f:[0,T]\times\Omega\times\real^2 \longrightarrow \real$ is an $\F$-progressively measurable process where as usual the $\omega$-dependence is omitted.

\subsection{Regularity of BSDE \eqref{chapBioedsrslips}: an approach inspired by \cite{ELKHuang, wang_ran_chen}}\label{sto:lip:1ereapproche}
 We consider the following assumption for $p>\frac12$ and $\beta >0$,
 \paragraph{Assumption (sL$^{p,\beta}$).}\text{ }
 
 \begin{itemize}
 \item[(i)] There exists two nonnegative $\mathbb F$-adapted processes $r$ and $\theta$ such that
 \begin{equation*}
 |f(t,y,z)-f(t,y',z')|\leq r_t |y-y'|+\theta_t |z-z'|, \; \forall (t,y,y',z,z')\in [0,T]\times \mathbb R^4.
 \end{equation*} 
 \item[(ii)] Let $a_t^2:= r_t+|\theta_t|^2$ for any $t\in [0,T]$. We suppose that $a_t^2>0$, $dt\otimes d\mathbb P$-a.e.,  $\mathbb E\left[A_T^a\right]<+\infty$ and
 \begin{equation*}
 \frac{f(t,0,0)}{a_t} \in \mathbb H_{2p,\beta, a}.
 \end{equation*}
 \item[(iii)] $\xi$ satisfies 
 $$\mathbb E\left[ e^{p\beta A^a_T} |\xi|^{2p}\right]<+\infty. $$
 \item[(iv)] If $p\in (\frac 12, 1)$, there exists a positive constant $L$ such that $A^a_T<L,$ $\mathbb P$-a.s.
 \end{itemize}
 \begin{Remark}
Notice that the case $a\equiv 0$ is excluded according to $(ii)$. However, this case can be studied easily since $a\equiv 0$ implies that $f$ is constant with respect to $y$ and $z$. Then, we can provide an explicit expression for the solution to this kind of BSDE.
 \end{Remark}
 The main difficulty in this study is that the process $a$ is not bounded and the stochastic integral of $a$ is not a BMO-martingale under Assumption \textbf{(sL$^{p,\beta}$)}. We recall the following result which can be found in \cite{wang_ran_chen} and extends the results in \cite{ELKHuang}.
 \begin{Theorem}[Theorem 4.1 together with Proposition 3.6 in \cite{wang_ran_chen}]\label{chapBioexistence:slipsch}
Let $p>\frac12$ and $\beta> \text{max } \{2/(2p-1); 3\}$ and assume that Assumption $(\mathbf{sL^{p,\beta}})$ holds. Then BSDE \eqref{chapBioedsrslips} admits a unique solution $(Y,Z)$ in $(\mathbb S_{2p,\beta} \cap \mathbb H_{2p,\beta}^a) \times \mathbb H_{2p,\beta}$. Moreover, 
\begin{itemize}
\item[$(i)$] if $p\geq 1$, there exists a positive constant $C_{p,\beta}$ depending only on $p$ and $\beta$ such that
\begin{equation}\label{chapBioestimatesL}
\| Y\|_{\mathbb S_{2p,\beta, a}}^{2p}+ \| Y\|_{\mathbb H^a_{2p,\beta, a}}^{2p}+\| Z\|_{\mathbb H_{2p,\beta, a}}^{2p}\leq C_{p,\beta}\left(\mathbb E\left[e^{p\beta A^a_T}|\xi|^{2p}\right]+  \left\|\frac{f(t,0,0)}{a_t} \right\|_{\mathbb H_{2p,\beta, a}}^{2p} \right),
\end{equation}
\item[$(ii)$] if $p\in(\frac12, 1)$, there exists a positive constant $C_{p,\beta, L}$ depending only on $p, \beta, L$ such that
Estimate \eqref{chapBioestimatesL} holds with $C_{p,\beta, L}$.
\end{itemize}
 \end{Theorem}
 \noindent We now turn to the Malliavin differentiability of solutions to BSDE \eqref{chapBioedsrslips} under Assumption $($\textbf{sL$^{p,\beta}$} $)$. Such a result requires additional assumptions that we now list.

 \paragraph{Assumption $\mathbf{(DsL^{p,\beta}})$.}There exist $p>\frac12$ and $\beta >0$ such that for any $h\in H$,

\vspace{0.5em}
\noindent  $(i)$ $\xi\in \mathbb D^{1,2}$,
 $$\lim\limits_{\varepsilon \to 0} \mathbb E\left[e^{p\beta A^a_T} \left|\frac{\xi\circ\tau_{\varepsilon h}-\xi}{\varepsilon}-\langle \nabla\xi,  h\rangle_{H} \right|^{2p} \right] =0,$$
 and
 $$ \mathbb E\left[e^{\beta A^a_T}  |\langle \nabla\xi,  h\rangle_{H} |^{2} \right] <+\infty.$$
 
\noindent $(ii)$ $\omega\longmapsto f(t,\omega,y,z)\in \mathbb D^{1,2}$ for any $(t,y,z)\in [0,T]\times \mathbb R\times \mathbb R$,
 $$\lim\limits_{\varepsilon \to 0} \left\|\frac{\dfrac{f(t,\omega\circ \tau_{\varepsilon h},Y_t,Z_t)-f(t,\omega, Y_t, Z_t)}{\varepsilon}- \langle \nabla f(t,Y_t, Z_t), h \rangle_H}{a_t}\right\|_{\mathbb H_{2p,\beta, a}}=0$$
 and
 $$ \left\| \frac{\langle \nabla f(t,Y_t, Z_t), h \rangle_H}{a_t}\right\|_{\mathbb H_{2,\beta, a}}<+\infty.$$

\noindent $(iii)$ Let $(\varepsilon_n)_{n\in \mathbb N}$ be a sequence in $(0,1]$ such that $\lim\limits_{n\to +\infty} \varepsilon_n=0$, and let  $(Y^n,Z^n)_n$ be a sequence of random variables which converges in $\mathbb S_{2p,\beta, a}\times \mathbb H_{2p,\beta, a}$, for any $(p,\beta)\in(\frac12,1)\times (0,+\infty)$, to some $(Y,Z)$. Then there exists $\eta>0$ such that for all $h\in H$, the following convergences hold in probability
\begin{align}
\label{chapBioassumpcts1}
&\|f_y(\cdot,\omega+\varepsilon_n h,Y^n_\cdot,Z_\cdot) -f_y(\cdot,\omega,Y_\cdot,Z_\cdot)\|_{L^{2+\eta}([0,T])}\underset{n\to+\infty}{\longrightarrow}0, \nonumber\\
&\underset{t\in [0,T]}{\text{ess sup}} \left| \frac{f_z(\cdot,\omega+\varepsilon_n h,Y^n_\cdot,Z^n_\cdot) -f_z(\cdot,\omega,Y_\cdot,Z_\cdot)}{a_t}\right| \underset{n\to+\infty}{\longrightarrow}0
\end{align}
or
\begin{align}
\label{chapBioassumpcts2}
&\|f_y(\cdot,\omega+\varepsilon_n h,Y^n_\cdot,Z^n_\cdot) -f_y(\cdot,\omega,Y_\cdot,Z_\cdot)\|_{L^{2+\eta}([0,T])}\underset{n\to+\infty}{\longrightarrow}0, \nonumber\\
&\underset{t\in [0,T]}{\text{ess sup}} \left| \frac{f_z(\cdot,\omega+\varepsilon_n h,Y_\cdot,Z^n_\cdot) -f_z(\cdot,\omega,Y_\cdot,Z_\cdot)}{a_t}\right| \underset{n\to+\infty}{\longrightarrow}0.
\end{align}

\noindent $(iv)$ For any $q\geq 1$, $\mathbb E\left[ \left(\int_0^T r_s ds\right)^q\right]<+\infty.$

\begin{Remark}\label{rk:pbDF}
Concerning Property $(ii)$ of Assumption $(\mathbf{DsL^{p,\beta}})$, notice that for fixed $(y,z)$, the process $(s,\omega) \longmapsto Df(s,\omega,y ,z)$ is defined outside a $\mathbb P$-negligible set which depends generally on $(y,z)$. Hence, it is not clear\footnote{This gap was pointed by Laurent Denis, during a review of the PhD thesis of the author, concerning Assumption $(D)$ in \cite{MastroliaPossamaiReveillac_MDBSDE} which corresponds to $(\mathbf{DsL^{p,\beta}})$. Remark \ref{rk:pbDF} has to be also taken into account for the latter paper. } that this process is well-defined at the point $(Y_s(\omega),Z_s(\omega))$. However, under appropriate continuity conditions on the map $(y,z)\longmapsto Df(s,\cdot, y,z)$, these negligible sets can actually be aggregated into a universal one, outside of which $Df(s,Y_s, Z_s)$ is indeed well-defined.\vspace{0.3em}

\noindent Nonetheless, let us point out an alternative approach for which no extra conditions on the Malliavin derivative of $f$ is required. The main problem is that the Malliavin derivative of a random variable is in general only defined $\mathbb P$-a.s. (except for instance when it is a cylindrical random variable), as a limit in probability of a sequence of random variables (which are defined for every $\omega$, again since they are cylindrical functions). There exists however a notion of limit, called medial limits $($lim med for short$)$, which has the particular property that under very general set theoretic axioms (see below), we have the following result (see \textit{e.g.} \cite{Moko}):
\vspace{0.3em}

\noindent Let $(Z_n)$ be a sequence of random variables, then $Z(\omega):= \underset{n\to +\infty}{\text{lim med}}\; Z_n(\omega)$ is universally measurable and if $Z_n$ converges to some random variable $Z^\mathbb P$ in probability, then $Z=Z^\mathbb P$, $\mathbb P$-a.s.\vspace{0.3em}

\noindent In our case, let $F$ be in $\mathbb D^{1,2}$, there exists a sequence of cylindrical elements $F_n$ which converges to $F$ in $\mathbb D^{1,2}$. Hence, $DF^n$ converges in $L^2(H)$ to the Malliavin derivative of $F$ denoted by $DF$, defined $\mathbb P$-a.s. Let $\widetilde{DF}$ be the medial limit of $DF^n$, defined for every $\omega$. By the above result, $DF=\widetilde{DF},\; \mathbb P-a.s.$ \vspace{0.3em}

\noindent This approach, which as far as we know has not been considered in the context of Malliavin calculus before (but see \cite{nutz} for its use for stochastic integrals), allows to give a complete pathwise definition of the Malliavin derivative of any random variable in $\mathbb D^{1,2}$. We emphasize nonetheless that the existence of medial limits depends on set-theoretic framework that one is using for instance Zermelo-Fraenkel set theory, plus the axiom of choice (ZFC for short), and either the continuum hypothesis or Martin's axiom (which is compatible with the negation of the continuum hypothesis). See \textit{e.g.} the footnote in \cite[Remark 4.1]{ptz} for more explanations and the weakest known conditions ensuring the existence of medial limits.
\end{Remark}

\noindent Before going further, we compare these assumptions with those made in \cite{MastroliaPossamaiReveillac_MDBSDE}. Assumptions $\mathbf{(DsL^{p,\beta})}$ $(i)$ and $(ii)$ seem quite reasonable in order to prove that the Malliavin derivatives of $Y_t$ and $Z_t$ are well-defined as the solution in $\mathbb S_{2,\beta, a}\times\mathbb H_{2,\beta, a}$ to the stochastic linear BSDE \eqref{chapBioeq:DY:thm} below, in view of Theorem \ref{chapBioexistence:slipsch}. We now turn to Assumption $\mathbf{(DsL^{p,\beta})}$ $(iii)$ which is less natural and stronger than its equivalent $(H_2)$ in \cite{MastroliaPossamaiReveillac_MDBSDE}. Indeed, if we compare for instance \eqref{chapBioassumpcts1} with its equivalent $(H_2)$ in \cite{MastroliaPossamaiReveillac_MDBSDE}, we first notice that we assume that \textbf{there exists $\boldsymbol{\eta>0}$} such that
$$\|f_y(\cdot,\omega+\varepsilon_n h,Y^n_\cdot,Z_\cdot) -f_y(\cdot,\omega,Y_\cdot,Z_\cdot)\|_{\mathbf{L^{2\boldsymbol{+\eta}}([0,T])}}\underset{n\to+\infty}{\longrightarrow}0,$$ which provide a condition of order strictly more than 2, unlike Assumption $(H_2)$ in \cite{MastroliaPossamaiReveillac_MDBSDE} which deals with an $L^2$ norm. This assumption is necessary for our study and comes in fact directly from the {\it a priori} estimates in Theorem \ref{chapBioexistence:slipsch} (see \cite[Proposition 3.6]{wang_ran_chen}) and the definition of $\mathbb H_{2p,\beta, a}$. We now turn to the second assumption in \eqref{chapBioassumpcts1}. This assumption is quite strong, and is intrinsically linked to the fact $Z\in \mathbb H_{2p,\beta, a}$. Indeed, to obtain \eqref{chapBioeq:passuper} in the proof of the Theorem \ref{chapBiothm:diff:sl} below, we are not able to conclude without this assumption since an H\"older Inequality will provide a term with $Z_s^{2+\eta}$ in the integral and in view of the definition of the space $\mathbb H_{2p,\beta, a}$, we can not prove the convergence. Concerning $(iv)$, this assumption is quite similar to those obtained in the following Section \ref{sto:lip:2emeapproche}, and is satisfied as soon as the stochastic integral of $r$ is for instance a BMO-martingale.\vspace{0.5em}

\noindent We thus have the following theorem.
 
 \begin{Theorem}\label{chapBiothm:diff:sl} Let $p$ be in $\in\left(\frac12, 1\right)$, $\beta> \text{max } \{2/(2p-1); 3\}$ and assume that Assumptions $\mathbf{(sL^{1,\beta})}$ and $\mathbf{(DsL^{p,\beta})}$ hold. Then, for any $t\in [0,T]$, $Y_t\in \mathbb D^{1,2}$ and $Z\in L^2([t,T]; \mathbb D^{1,2}) $. Besides, a version of 
$$(D_u Y_t, D_u Z_t)_{0\leq u\leq t, 0\leq t\leq T},$$ is given as the solution to the affine BSDE:
\begin{align}
\nonumber
D_u Y_t &= D_u \xi +\int_t^T \left(D_u f(s,Y_s,Z_s) +f_y(s,Y_s,Z_s)D_u Y_s+ f_z(s,Y_s,Z_s) D_u Z_s\right)ds \\
\label{chapBioeq:DY:thm}&-\int_t^T D_u Z_s dW_s.
\end{align}
\end{Theorem}
 \begin{proof}  We only consider the case where \eqref{chapBioassumpcts1} holds under Assumption $\mathbf{(DsL^{p,\beta}})$ $(iii)$, since the other situation can be treated similarly. We aim at applying Theorem \ref{thm:newcarD12} with $F=Y_t$ and $F=\int_t^T Z_s dW_s$. The proof is similar to the proof of Theorem 5.1 in \cite{MastroliaPossamaiReveillac_MDBSDE} and we recall here the main ideas. Let $\varepsilon >0$, $h\in H$ and $p\in \left(\frac12, 1\right)$. We have 
 $$ Y_s\circ \tau_{\varepsilon h} =\xi\circ\tau_{\varepsilon h} +\int_s^T f(r,Y_r,Z_r) \circ\tau_{\varepsilon h} dr -\int_s^T Z_r\circ\tau_{\varepsilon h} dW_r, \quad \forall s \in [t,T], \; \P_0-a.s.  $$
As a consequence, setting for simplicity
$$Y_s^\varepsilon:=\frac1\varepsilon (Y_s\circ \tau_{\varepsilon h}-Y_s),\quad Z_s^\varepsilon:=\frac1\varepsilon (Z_s\circ \tau_{\varepsilon h}-Z_s), \quad \xi^\varepsilon:=\frac1\varepsilon (\xi\circ \tau_{\varepsilon h}-\xi), \; s \in [t,T],$$ 
we have that $(Y^\varepsilon, Z^\varepsilon)$ solves the BSDE:
\begin{equation}\label{chapBioedsr_accroissement}
Y_s^\varepsilon=\xi^\varepsilon+ \int_s^T\left( \tilde{A}_r^\varepsilon+\tilde{A}_r^{y,\varepsilon} Y_r^\varepsilon+\tilde{A}_r^{z,\varepsilon} Z_r^\varepsilon\right) dr-\int_s^T Z_r^\varepsilon dW_r,
\end{equation}
with
\begin{align*}
    &\tilde{A}_r^{y,\varepsilon}:=\int_0^1 f_y(r,\cdot+\varepsilon h, Y_r+\theta (Y_r\circ \tau_{\varepsilon h}-Y_r), Z_r)d\theta,\\
  &\tilde{A}_r^{z,\varepsilon}:= \int_0^1f_z(r,\cdot+\varepsilon h, Y_r\circ \tau_{\varepsilon h}, Z_r+\theta (Z_r\circ \tau_{\varepsilon h}-Z_r)) d\theta,\\
  & \tilde{A}_r^\varepsilon:=\frac1\varepsilon (f(r,\cdot+\varepsilon h, Y_r, Z_r)-f(r,\cdot,Y_r,Z_r)).
 \end{align*}

\vspace{0.5em}
\noindent Let us now consider the following stochastic affine BSDE on $[t,T]$, which admits a unique solution $(\tilde{Y}^h, \tilde{Z}^h)\in (\mathbb S_{2,\beta,a} \cap \mathbb H_{2,\beta, a}^a) \times \mathbb H_{2,\beta, a}$ according to Theorem \ref{chapBioexistence:slipsch} under Assumption $\mathbf{(DsL^{p,\beta}})$
\begin{align}
\label{chapBioeq:DY,h}
\nonumber \tilde{Y}^h_s&=\langle D\xi,\dot h\rangle_{L^2([0,T])} -\int_s^T \tilde{Z}_r^h dW_r\\
&\hspace{1em}+\int_s^T \left(\langle Df(r,Y_r,Z_r),\dot h\rangle_{L^2([0,T])}+\tilde{Y}_r^hf_y(r,Y_r,Z_r)+\tilde{Z}_r^hf_z(r,Y_r,Z_r) \right)dr.
\end{align}
Hence, using Theorem \ref{chapBioexistence:slipsch} together with Inequality \eqref{inegalite:norme}, we obtain
\begin{align*}
\| Y^\varepsilon- \tilde{Y}^h\|_{\mathcal S^{2p}}^{2p}+\|Z^\varepsilon-\tilde{Z}^h \|_{\mathcal H^{2p} }^{2p}&\leq \| Y^\varepsilon- \tilde{Y}^h\|_{\mathbb S_{2p,\beta, a}}^{2p}+\|Z^\varepsilon-\tilde{Z}^h \|_{\mathbb H_{2p,\beta, a} }^{2p} \\
&\leq C_{1,\beta}\left( \Xi^{p,a,\beta}_\varepsilon +X^{\varepsilon}_T+X^{y,\varepsilon}_T+ X^{z,\varepsilon}_T\right)
\end{align*}
where
\begin{align*}
\Xi^{p,a,\beta}_\varepsilon&:= \mathbb E\left[ e^{p\beta A^a_T} \left|\xi^\varepsilon-\langle \nabla\xi,  h\rangle_{H} \right|^{2p}\right],\ X^\varepsilon_T:=\left\|\frac{\tilde A_t^\varepsilon- \langle \nabla f(t,Y_t, Z_t), h \rangle_H}{a_t}\right\|^{2p}_{\mathbb H_{2p,\beta, a}},\\
X^{y,\varepsilon}_T&:= \left\|\tilde{Y}^h_t\frac{\tilde A_t^{y,\varepsilon}- f_y(t,Y_t,Z_t)}{a_t}\right\|^{2p}_{\mathbb H_{2p,\beta, a}},\ X^{z,\varepsilon}_T:=\left\|\tilde{Z}^h_t\frac{\tilde A_t^{z,\varepsilon}- f_z(t,Y_t,Z_t)}{a_t}\right\|^{2p}_{\mathbb H_{2p,\beta, a}}.
\end{align*}
First notice that under Assumption $\mathbf{(DsL^{p,\beta}})$ $(i)$ and $(ii)$, we have 
\begin{equation}\label{chapBioestimation1}
 \lim\limits_{\varepsilon \to 0}\left(\Xi_\varepsilon^{p,a,\beta} +X^{\varepsilon}_T\right) =0.
\end{equation}
We now turn to $X_T^{y,\varepsilon}$. We have
\begin{align*}
X_T^{y,\varepsilon}&=\mathbb E\left[ \left( \int_0^T e^{\beta A^a_t} |\tilde{Y}^h_t|^2\left|\frac{ A_t^{y,\varepsilon}- f_y(t,Y_t,Z_t)}{a_t} \right|^2 dt\right)^p \right].
\end{align*}
According to Assumption $\mathbf{(DsL^{p,\beta}})$ $(iii)$, there exists $\eta>0$ such that
$$\left\| \frac{ A_t^{y,\varepsilon}- f_y(t,Y_t,Z_t)}{a_t}\right\|_{L^{2+\eta}([0,T])}^{2+\eta}=\int_0^T \left|\frac{ A_t^{y,\varepsilon}- f_y(t,Y_t,Z_t)}{a_t}\right|^{2+\eta}dt\overset{\text{ proba }}{ \underset{\varepsilon \to 0}{\longrightarrow} }0.$$
Hence, using H\"older Inequality with $q>1$ such that $2q=2+\eta$ and denoting by $\overline q$ its conjugate and using the fact that $\tilde{Y}^h\in \mathbb S_{2,\beta, a}$, we have for some constant $C>0$
\begin{align*}
& \int_0^T e^{\beta A^a_t} |\tilde{Y}^h_t|^2\left|\frac{ A_t^{y,\varepsilon}- f_y(t,Y_t,Z_t)}{a_t} \right|^2 dt\\
&  \leq C\left( \int_0^T e^{\overline q\beta A^a_t} |\tilde{Y}^h_t|^{2\overline q} dt \right)^{1/\overline q}  \left\| \frac{ A_t^{y,\varepsilon}- f_y(t,Y_t,Z_t)}{a_t}\right\|_{L^{2+\eta}([0,T])}\\
&\leq C \| \tilde{Y}^h\|_{\mathbb S_{2,\beta}}^2  \left\| \frac{ A_t^{y,\varepsilon}- f_y(t,Y_t,Z_t)}{a_t}\right\|_{L^{2+\eta}([0,T])}\\
&\underset{\varepsilon \to 0 }\longrightarrow 0, \text{ in probability}.
\end{align*}
Now, let $\eta>0$ small enough such that $2(p+\eta)\in (1,2)$. Then, by noticing that there exists a positive constant $c$, such that $\left|\frac{ A_t^{y,\varepsilon}- f_y(t,Y_t,Z_t)}{a_t} \right|^2 \leq c r_t$, since $|f_y(t,y,z)|\leq r_t$ for any $(t,y,z)\in [0,T]\times \mathbb R^2$ and from $(iv)$, there exists a positive constant $C$ such that
\begin{align*}
&\underset{\varepsilon \in (0,1)}\sup \mathbb E\left[ \left( \int_0^T e^{\beta A^a_t} |\tilde{Y}^h_t|^2\left|\frac{ A_t^{y,\varepsilon}- f_y(t,Y_t,Z_t)}{a_t} \right|^2 dt\right)^{p+\eta} \right]\\
&\leq  C \mathbb E\left[ \underset{t\in [0,T]}{\sup} e^{(p+\eta)\beta A^a_t} |\tilde{Y}^h_t|^{2(p+\eta)}\right]\\
&<+\infty,
\end{align*}
since $2(p+\eta)<2$ and $\tilde{Y}^h\in \mathbb S_{2,\beta, a}$. Hence, using de La Vall\'ee-Poussin Criterion, we deduce that the family of random variables 
$$\left\{\left( \int_0^T e^{\beta A^a_t} |\tilde{Y}^h_t|^2\left|\frac{ A_t^{y,\varepsilon}- f_y(t,Y_t,Z_t)}{a_t} \right|^2 dt\right)^{p}\right\},  \; \varepsilon\in (0,1),$$ is uniformly integrable. Hence, by the dominated convergence theorem, we deduce that
\begin{equation}\label{chapBioestimation2}
X_T^{y,\varepsilon} \underset{\varepsilon \to 0}{\longrightarrow} 0.
\end{equation}
We now turn to $X_T^{z,\varepsilon}$. By proceeding similarly, we have
\begin{align*}
X_T^{z,\varepsilon}&=\mathbb E\left[ \left( \int_0^T e^{\beta A^a_t} |\tilde{Z}^h_t|^2\left|\frac{ A_t^{z,\varepsilon}- f_z(t,Y_t,Z_t)}{a_t} \right|^2 dt\right)^p \right].
\end{align*}
According to Assumption $\mathbf{(DsL^{p,\beta}})$ $(iii)$ and using the fact that $\tilde{Z}^h\in \mathbb H_{2,\beta, a}$ we know that for any $t\in [0,T]$
\begin{equation}
\label{chapBioeq:passuper} \int_0^T e^{\beta A^a_t} |\tilde{Z}^h_t|^2\left|\frac{ A_t^{z,\varepsilon}- f_z(t,Y_t,Z_t)}{a_t} \right|^2 dt\overset{\text{ proba }}{ \underset{\varepsilon \to 0}{\longrightarrow} }0.
\end{equation}
Let $\eta>0$ small enough such that $2(p+\eta)\in (1,2)$. Then, we can show similarly that there exists a positive constant $C$ such that
\begin{align*}
&\underset{\varepsilon \in (0,1)}\sup \mathbb E\left[ \left( \int_0^T e^{\beta A^a_t} |\tilde{Z}^h_t|^2\left|\frac{ A_t^{z,\varepsilon}- f_z(t,Y_t,Z_t)}{a_t} \right|^2 dt\right)^{p+\eta} \right]\\
&\leq  C \mathbb E\left[  \left( \int_0^T e^{\beta A^a_t} |\tilde{Z}^h_t|^2 dt\right)^{p+\eta}\right]\\
&<+\infty,
\end{align*}
since $2(p+\eta)<2$ and $\tilde{Z}^h\in \mathbb H_{2,\beta, a}$. Hence, using de La Vall\'ee-Poussin Criterion, \eqref{chapBioeq:passuper} and the dominated convergence theorem, we deduce that
\begin{equation}\label{chapBioestimation3}
X_T^{z,\varepsilon} \underset{\varepsilon \to 0}{\longrightarrow} 0.
\end{equation}
Finally, from \eqref{chapBioestimation1}, \eqref{chapBioestimation2} and \eqref{chapBioestimation3}, we thus obtain for $p\in \left(\frac12,1 \right)$
$$\| Y^\varepsilon- \tilde{Y}^h\|_{\mathcal S^{2p}}^{2p}+\|Z^\varepsilon-\tilde{Z}^h \|_{\mathcal H^{2p} }^{2p}\underset{\varepsilon \to 0}{\longrightarrow} 0.$$
The rest of the proof is then similar to the proof of Theorem 5.1 in \cite{MastroliaPossamaiReveillac_MDBSDE} and by applying Theorem \ref{thm:newcarD12}, we deduce that $Y_t\in \mathbb D^{1,2}$ and using \cite[Lemma 2.3]{PardouxPeng_BSDEQuasilinearParabolicPDE}, one shows that $Z$ belongs to $L^2([t,T]; \mathbb D^{1,2})$. Besides, we can prove that a version of 
$$(D_u Y_t, D_u Z_t)_{0\leq u\leq t, 0\leq t\leq T},$$ is given as the solution to the affine BSDE:
\begin{align*}
\nonumber D_u Y_t& = D_u \xi +\int_t^T \left(D_u f(s,Y_s,Z_s) +f_y(s,Y_s,Z_s)D_uY_s+ f_z(s,Y_s,Z_s) D_u Z_s\right)ds\\
& -\int_t^T D_u Z_s dW_s,
\end{align*}
which admits, from Assumption $\mathbf{(DsL^{p,\beta}})$ and Theorem \ref{chapBioexistence:slipsch}, or \cite{ELKHuang}, a unique solution in $\mathbb S_{2,\beta, a}\times \mathbb H_{2,\beta, a}$.
 \end{proof}
 
 \subsection{Regularity of BSDE \eqref{chapBioedsrslips}: an approach inspired by \cite{AIR, BriandConfortola}}\label{sto:lip:2emeapproche} In this section, we will study the Malliavin differentiability of BSDE \eqref{chapBioedsrslips} in the stochastic Lipschitz case by using the theory developed in \cite{BriandConfortola}. A similar theory, using the BMO theory was also developed in \cite{AIR} but for particular stochastic Lipschitz BSDE (see BSDE (16) in \cite[Section 4]{AIR}). We recall Assumptions A1. and A2. from \cite{BriandConfortola}.

\vspace{0.5em}
\noindent $\mathbf{(BC_1)}$ Assume that there exists a real predictable process $K$ bounded from below by $1$ and a constant $\alpha\in (0,1)$ such that
 \begin{itemize}
 \item[$(i)$] For each $t\in [0,T]$, $(y,z)\longmapsto f(t,y,z)$ is continuous,
 \item[$(ii)$] For any $(t,y,y',z,z')\in [0,T]\times \mathbb R^2\times (L^2([0,T]))^2 $ ,
 $$(y-y')(f(t,y,z)-f(t,y',z))\leq K_t^{2\alpha}|y-y'|^2,$$
 and
 $$|f(t,y,z)-f(t,y,z')|\leq K_t\|z-z'\|_{L^2([0,T])}. $$
 \item[$(iii)$] There exists a constant $C>0$ such that for any stopping time $\tau\leq T$:
 $$\mathbb E\left[ \int_\tau^T |K_s|^2 ds \Big| \mathcal F_\tau\right]\leq C^2. $$ 
 \noindent We denote by $N$ the smallest constant $C$ which satisfies this statement.
 \end{itemize}

\noindent Notice that if the previous Assumption $\mathbf{(BC_1)}$ $(iii)$ is satisfied then for any $u\in L^2([0,T])$ with $1=\| u\|_{L^2([0,T])}$, $\left(M_t:= \int_0^t K_s u_sdW_s\right)_{t\in [0,T]}$ is a BMO-martingale and $\|M \|_{\text{BMO}}=N$.
Let now
$$\Phi(p):= \left( 1+\frac{1}{p^2} \log\left( \frac{2p-1}{2(p-1)}\right)\right)^{\frac12}-1, $$
and $q^\star$ such that $\Phi(q^\star)=N$. We then defined $p_\star$ the conjugate of $q^\star$, defined by
$$ \frac1{q^\star}+ \frac1{p_\star}=1.$$ We now recall Assumption A3. and A4. of \cite{BriandConfortola}.

\vspace{0.5em}
\noindent $\mathbf{(BC_2)}$ There exists $p^\star>p_\star>1$ such that 
 $$\mathbb E\left[|\xi|^{p^\star}+\left(\int_0^T |f(s,0,0)| ds\right)^{p^\star} \right]<+\infty. $$
$\mathbf{(BC_3)}$ There exists a non negative predictable process $F$ such that 
 $$\mathbb E\left[\left(\int_0^T |F_s| ds\right)^{p^\star} \right]<+\infty, $$
 and
 $$ \forall (t,y,z)\in [0,T]\times \mathbb R\times \mathbb R, \; |f(t,y,z)|\leq F_t +K_t^{2\alpha}|y|+K_t|z|, \,\mathbb P-a.s.$$
 Then, we have the following {\it a priori} estimates for solutions to BSDE \eqref{chapBioedsrslips}.
 \begin{Theorem}[see Corollary 3.4 and Theorem 3.5 in \cite{BriandConfortola}]\label{Mbio:thm:apriori:BC}
 Assume that Assumptions $\mathbf{(BC_1)},$ $\mathbf{(BC_2)}$ and $\mathbf{(BC_3)}$ hold. Then, BSDE \eqref{chapBioedsrslips} admits a unique solution $(Y,Z)\in \mathcal S_p\times \mathcal H_p$ for any $p<p^\star$. Besides, for each $p\in (p_\star,p^\star)$,
 \begin{align}
\nonumber \| Y\|_{\mathcal S_p}+\| Z\|_{\mathcal H_p} &\leq C\mathbb E\left[|\xi|^{p^\star}+\left(\int_0^T |f(s,0,0)| ds\right)^{p^\star} \right]^\frac1{p^\star}\\
\label{apriori:briandconfortola} &\hspace{1em}\times \left( 1+\mathbb E\left[\left( \int_0^T K_s^{2\alpha}+K_s^2 ds\right)^\frac{P^\star}{2} \right]^{\frac{1}{P^\star}}\right),
 \end{align} 
 where $P^\star=p(p^\star+p)/(p^\star-p)$ and $C$ is a positive constant.
 \end{Theorem}
 \noindent We now set the following assumptions

\vspace{0.5em}
\noindent $\mathbf{(sD^\infty)}$ For any $p>1$, $\xi$ belongs to $\mathbb D^{1,p}$ and $\omega \longmapsto f(t,\omega,y,z)$ belongs to $L^2([0,T]; \mathbb D^{1,p})$.

\vspace{0.5em}
\noindent ($\mathbf{sH_{1,\infty}}$) For any $p>1$ and for any $h\in H$
$$\hspace{-1em}\lim\limits_{\varepsilon \to 0}\ \E\left[ \left(\int_0^T\left|\frac{f(s,\cdot+\varepsilon h, Y_s, Z_s)-f(s,\cdot,Y_s,Z_s)}{\varepsilon}-\langle Df(s,\cdot,Y_s,Z_s),\dot h\rangle_{\mathfrak h}\right| ds\right)^p\right]=0.$$

\vspace{0.5em}
\noindent ($\mathbf{sH_{2,\infty}}$) Let $(\varepsilon_k)_{k\in \mathbb N}$ be a sequence in $(0,1]$ such that $\lim\limits_{k\to +\infty} \varepsilon_k=0$, and let  $(Y^k,Z^k)_k$ be a sequence of random variables which converges in $\S_p\times \H_p$ for any $p<p^*$ to some $(Y,Z)$. Then for all $h\in H$, the following convergences hold in probability
\begin{align}
\label{chap2assumpcts1SBC}
&\|f_y(\cdot,\omega+\varepsilon_k h,Y^k_\cdot,Z_\cdot) -f_y(\cdot,\omega,Y_\cdot,Z_\cdot)\|_{L^2([0,T])}\underset{k\to+\infty}{\longrightarrow}0 \nonumber\\
& \|f_z(\cdot,\omega+\varepsilon_k h,Y^k_\cdot,Z^k_\cdot)-  f_z(\cdot,\omega,Y_\cdot,Z_\cdot)\|_{L^2([0,T])}\underset{k\to+\infty}{\longrightarrow}0,\end{align}
or
\begin{align}
\label{chap2assumpcts2SBC}
\|f_y(\cdot,\omega+\varepsilon_k h,Y^k_\cdot,Z^k_\cdot) -f_y(\cdot,\omega,Y_\cdot,Z_\cdot)\|_{L^2([0,T])}\underset{k\to +\infty}{\longrightarrow}0\nonumber\\
 \|f_z(\cdot,\omega+\varepsilon_k h,Y_\cdot,Z^k_\cdot)   -f_z(\cdot,\omega,Y_\cdot,Z_\cdot)\|_{L^2([0,T])}\underset{k\to +\infty}{\longrightarrow}0.
\end{align}

\begin{Remark}\label{remark:pourtoutp}
Notice that Assumption $\mathbf{(sH_{1,\infty})}$ implies that both $\mathbf{(BC2)}$ and $\mathbf{(BC3)}$ are true for any $p^*>1$. Thus, Theorem \ref{Mbio:thm:apriori:BC} holds under $\mathbf{(BC1)}$ and $(\mathbf{sH_{1,\infty}})$ and Inequality \eqref{apriori:briandconfortola} is satisfied for any $p>1$ with a corresponding $p^*$ which can be chosen greater than $p_*$ defined by $\mathbf{(BC1)}$.
\end{Remark} 
\noindent We can now state the main result of this section.
\begin{Theorem}\label{chapBiothm:diff:sl:BC} Assume that $(\mathbf{BC_1})$-$(\mathbf{BC_3})$, $(\mathbf{ D^{1,\infty}})$, $(\mathbf{sH_{1,\infty}})$ and $(\mathbf{sH_{2,\infty}})$ hold. Then, for any $p>1$ and $t\in [0,T]$, $Y_t\in \mathbb D^{1,p}$ and $Z\in L^2([t,T]; \mathbb D^{1,p}) $. Besides, a version of 
$$(D_u Y_t, D_u Z_t)_{0\leq u\leq t, 0\leq t\leq T},$$ is given as the solution to the affine BSDE:
\begin{align}
\nonumber
D_u Y_t& = D_u\xi +\int_t^T \left(D_u f(s,Y_s,Z_s) +f_y(s,Y_s,Z_s)D_u Y_s+ f_z(s,Y_s,Z_s) D_u Z_s\right)ds\\
\label{chapBioeq:DY:thm:BC}& -\int_t^T D_u Z_s dW_s.
\end{align}
\end{Theorem}
\begin{proof}  The proof is similar to that of Theorem \ref{chapBiothm:diff:sl}. We only consider the case where \eqref{chap2assumpcts1SBC} holds in Assumption $(\mathbf{sH_{1,\infty}})$ since the other one can be treated similarly. First notice that under Assumption $\mathbf{(sD^\infty)}$, $(\mathbf{BC_1})$-$(\mathbf{BC_3})$ and according to Theorem \ref{Mbio:thm:apriori:BC} together with Remark \ref{remark:pourtoutp}, for any $p>1$,  $\| Y\|_{\mathcal S_p}+\| Z\|_{\mathcal H_p}<+\infty$   . We aim at applying Theorem \ref{thm:newcarD12} (see \cite{MastroliaPossamaiReveillac_MDBSDE}) with $F=Y_t$ and $F=\int_t^T Z_s dW_s$. The proof is very close to the proof of Theorem 5.1 in \cite{MastroliaPossamaiReveillac_MDBSDE} and we recall here the main ideas. Let $\varepsilon >0$ and $h\in H$. We have 
\begin{equation}
\label{Mbio:edsr:affine:epsilon}Y_s\circ \tau_{\varepsilon h} =\xi\circ\tau_{\varepsilon h} +\int_s^T f(r,Y_r,Z_r) \circ\tau_{\varepsilon h} dr -\int_s^T Z_r\circ\tau_{\varepsilon h} dW_r, \quad \forall s \in [t,T], \; \P_0-a.s. 
\end{equation}
As a consequence, setting for simplicity
$$Y_s^\varepsilon:=\frac1\varepsilon (Y_s\circ \tau_{\varepsilon h}-Y_s),\quad Z_s^\varepsilon:=\frac1\varepsilon (Z_s\circ \tau_{\varepsilon h}-Z_s), \quad \xi^\varepsilon:=\frac1\varepsilon (\xi\circ \tau_{\varepsilon h}-\xi), \; s \in [t,T],$$ 
we have that $(Y^\varepsilon, Z^\varepsilon)$ solves the BSDE:
\begin{equation}\label{chapBioedsr_accroissement:BC}
Y_s^\varepsilon=\xi^\varepsilon+ \int_s^T (\tilde{A}_r^\varepsilon+\tilde{A}_r^{y,\varepsilon} Y_r^\varepsilon+\tilde{A}_r^{z,\varepsilon} Z_r^\varepsilon )dr-\int_s^T Z_r^\varepsilon dW_r,
\end{equation}
with
\begin{align*}
    &\tilde{A}_r^{y,\varepsilon}:=\int_0^1 f_y(r,\cdot+\varepsilon h, Y_r+\theta (Y_r\circ \tau_{\varepsilon h}-Y_r), Z_r)d\theta,\\
  &\tilde{A}_r^{z,\varepsilon}:= \int_0^1f_z(r,\cdot+\varepsilon h, Y_r\circ \tau_{\varepsilon h}, Z_r+\theta (Z_r\circ \tau_{\varepsilon h}-Z_r)) d\theta,\\
  & \tilde{A}_r^\varepsilon:=\frac1\varepsilon (f(r,\cdot+\varepsilon h, Y_r, Z_r)-f(r,\cdot,Y_r,Z_r)).
 \end{align*}
Hence, under Assumptions  $(\mathbf{BC_1})$-$(\mathbf{BC_3})$, $(\mathbf{sD^{1,\infty}})$, according to Theorem \ref{Mbio:thm:apriori:BC}, $(Y^\varepsilon, Z^\varepsilon)$ is the unique solution of BSDE \eqref{Mbio:edsr:affine:epsilon} in $\mathcal S_p\times \mathcal H_p$ for any $p>1$.\vspace{0.5em}

\noindent Consider now the following stochastic affine BSDE on $[t,T]$, which admits a unique solution $(\tilde{Y}^h, \tilde{Z}^h)\in (\mathcal S_{p} \times\mathcal H_{p})$ for any $p>1$ according to Theorem \ref{Mbio:thm:apriori:BC} under Assumption  $(\mathbf{BC_1})$-$(\mathbf{BC_3})$, $(\mathbf{sD^{1,\infty}})$,
\begin{align}
\label{chapBioeq:DY,h:BC}
\nonumber \tilde{Y}^h_s&=\langle D\xi,\dot h\rangle_{L^2([0,T])}-\int_s^T \tilde{Z}_r^h dW_r\\
&+\int_s^T \left(\langle Df(r,Y_r,Z_r),\dot h\rangle_{L^2([0,T])}+\tilde{Y}_r^hf_y(r,Y_r,Z_r)+\tilde{Z}_r^hf_z(r,Y_r,Z_r) \right)dr.
\end{align}
Hence, using Theorem \ref{Mbio:thm:apriori:BC} we obtain for any $p^*>p>1$
\begin{align*}
&\| Y^\varepsilon- \tilde{Y}^h\|_{\mathcal S^{p}}+\|Z^\varepsilon-\tilde{Z}^h \|_{\mathcal H^{p} }\\
&\leq C\mathbb E\left[|\xi^\varepsilon-\langle D\xi,\dot h\rangle_{L^2([0,T])}|^{p^\star}+\left(\int_0^T |X_s^{\varepsilon}+ X_s^{y,\varepsilon}+X_s^{z,\varepsilon}| ds\right)^{p^\star} \right]^\frac1{p^\star}\\
&\hspace{1em}\times \left( 1+\mathbb E\left[\left( \int_0^T (K_s^{2\alpha}+K_s^2 )ds\right)^{P^\star/2} \right]^{1/{P^\star}}\right),
\end{align*}
where
\begin{align*}
X_s^\varepsilon&:=\tilde{A}_s^\varepsilon-  \langle Df(s,Y_s,Z_s),\dot h\rangle_{L^2([0,T])}\\
X^{y,\varepsilon}_s&:= \tilde{Y}_s^h(\tilde{A}_s^{y,\varepsilon}-f_y(s,Y_s,Z_s))\\
X^{z,\varepsilon}_s&:=\tilde{Z}_s^h(\tilde{A}_s^{z,\varepsilon}-f_z(s,Y_r,Z_r) ).
\end{align*}
Notice that under $\mathbf{(BC1)}$ $(iii)$ we have 
$$ \mathbb E\left[\left( \int_0^T (K_s^{2\alpha}+K_s^2 )ds\right)^{P^\star/2} \right]<+\infty.$$
Hence, after the same kind of calculations than those made in the proofs of Theorem 5.1 in \cite{MastroliaPossamaiReveillac_MDBSDE} or Theorem \ref{chapBiothm:diff:sl} above, we deduce that for any $t\in [0,T]$ and any $p>1$, $Y_t\in \mathbb D^{1,p}$ and $Z\in L^2([t,T]; \mathbb D^{1,p}) $ and that their Malliavin derivatives are solutions to \eqref{chapBioeq:DY:thm:BC}.
 \end{proof}
 
\subsection{Discussion and comparison of results}\label{Mbio:comparaison}
We begin with Assumption $\mathbf{(DsL^{p,\beta})}$ and the first approach inspired by \cite{wang_ran_chen}. Even if Assumption $(i)$ is not too restrictive in view of the theory developed in \cite{wang_ran_chen}, in practice we could have some difficulties to verify $(ii)$ and $(iii)$. Indeed, in $(ii)$ we have to control the norm in $\mathbb H_{2p,\beta,a}$ of $1/a_t$, and $(iii)$ requires a control of the ess sup of the derivative of $f$ with respect to $z$. As soon as $r$ and $\theta$ are random, these assumptions restrict significantly the range of possible applications. As explained above, these assumptions are strongly linked to \textit{{\it a priori}} estimates obtained in \cite{wang_ran_chen}, which suggests to modify the proofs in \cite{wang_ran_chen} to try to obtain weaker conditions, if possible.
\vspace{0.5em}

\noindent Concerning the second approach, \textit{{\it a priori}} estimates \eqref{apriori:briandconfortola} seem to be better, since they are similar to those obtained in the Lipschitz or quadratic case (see \cite{BriandDelyonHuPardouxStoica}). Notice however that the order of these {\it a priori} estimates depends closely on the BMO-norm of the stochastic integral of the Lipschitz constant $K$, which in practice, could be quite difficult to control. We provide conditions in $\mathbb D^{1,p}$ for any $p>1$ due to the control of the norm of $Y$ and $Z$ at an order depending on this BMO-norm. Assumptions $\mathbf{(sD^\infty)}$  and $\mathbf{(sH_{1,\infty})}$ are not so surprising, since they are similar to conditions obtained in Section 7 in \cite{MastroliaPossamaiReveillac_MDBSDE} when dealing with quadratic growth BSDEs.\vspace{0.5em}

\noindent From now, we set the following two assumptions.

\vspace{0.5em}
\noindent $\mathbf{(EKH^{p,\beta})}$ Let Assumptions $\mathbf{(sL^{1,\beta})}$ and $\mathbf{(DsL^{p,\beta})}$ hold.\vspace{0.5em}

\noindent $\mathbf{(BC)}$ Let Assumptions $(\mathbf{BC_1})$-$(\mathbf{BC_3})$, $(\mathbf{s D^{\infty}})$, $(\mathbf{sH_{1,\infty}})$ and $(\mathbf{sH_{2,\infty}})$ hold.

\subsection{Example: affine BSDE with unbounded coefficients}\label{sectionaffinestoBSDE}
We now study a particular stochastic Lipschitz BSDE in the non-Markovian case:
\begin{equation}\label{chapBioedsraffine}
Y_t=\xi+\int_t^Tf(s,Y_s,Z_s)ds -\int_t^T Z_s dW_s, \; \forall t\in [0,T], \; \mathbb P-a.s.
\end{equation}
where 
\begin{align*}
f: [0,T]\times\Omega\times \mathbb R\times\mathbb R &\longrightarrow \mathbb R\\
(t,\omega,y,z)&\longmapsto \lambda_s(\omega)+\mu_s(\omega) y+\nu_s(\omega) z,
\end{align*}
and where $\xi$ is an $\mathcal{F}_T$-measurable random variable and $\lambda,\mu,\nu:[0,T]\times\Omega \longrightarrow \real$ are $\F$-progressively measurable processes. 
\begin{Remark}The BSDE \eqref{chapBioedsraffine} studied in this section generalizes \cite[BSDE (5)]{AIR} for affine BSDEs, since the generator of \eqref{chapBioedsraffine} is affine in both $Y$ and $Z$. By adding a Lipschitz coefficient with respect to $Y$ and $Z$ in the generator satisfying Assumption $(A3)$ in \cite{AIR}, one could show that we strictly extend \cite[Section 3]{AIR}. Besides, we insist on the fact that $\lambda$, $\mu$ and $\nu$ are not bounded, which also extend the results in \cite[Section 6.2]{MastroliaPossamaiReveillac_MDBSDE}.
\end{Remark}

\noindent $\mathbf{(A_1)}$ There exists a constant $C>0$ such that for any stopping time $\tau\leq T$:
 $$\mathbb E\left[ \int_\tau^T |\nu_s|^2 ds \Big| \mathcal F_\tau\right]\leq C^2. $$

 \noindent $\mathbf{(A_2)}$ For any $p>1$,
 \begin{itemize}
 \item[$(i)$]$\exp\left( \int_0^\cdot |\mu_s| ds\right)\in \mathcal S_p, $
 and 
 \item[$(ii)$]$\mathbb E\left[ |\xi|^p +\left( \int_0^T |\lambda_t| dt\right)^p\right]<+\infty. $
 \end{itemize}

\noindent Before going further, notice that $\mathbf{(A_1)}$ is equivalent to saying that $\int_0^\cdot \nu_s dW_s$ is a BMO-martingale, which corresponds to Assumption $($A2$)$ in \cite{AIR} or Assumption A1. in \cite{BriandConfortola}. However, we do not assume that the same statement holds for $\int_0^\cdot \mu_s dW_s$. Indeed, in $\mathbf{(A_2)}$ we just assume that the process $\left( \int_0^\cdot \mu_s ds\right)_{t\in [0,T]}$ has exponential moments of all orders. 
 \begin{Theorem}\label{thm:affine:existence}
 Assume that Assumptions $\mathbf{(A_1)}$ and $\mathbf{(A_2)}$ hold. Then, BSDE \eqref{chapBioedsraffine} admits a unique solution $(Y,Z)\in \mathcal S_p\times \mathcal H_p$ for any $p>1$. Besides, Estimate \eqref{apriori:briandconfortola} holds for any $1<p<p^*$. 
 \end{Theorem}
 
 \begin{proof} $\mathbf{(A_1)}$ and $\mathbf{(A_2)} (i)$ are weaker assumptions than $\mathbf{(BC_1)}$, so we cannot apply directly Corollary 3.4 and Theorem 3.5 in \cite{BriandConfortola}. However, by reproducing the proof of Lemma 3.2, Lemma 3.3, Corollary 3.4 and Theorem 3.5 in \cite{BriandConfortola}, one notices that we only need to have a BMO-property for $\int_0^\cdot \nu_s dW_s$, since only Relation $(2)$ in \cite{BriandConfortola}, corresponding to $\mathbf{(A_2)} (i)$, is used to deal with terms depending on $\mu$. Hence, for affine BSDE \eqref{chapBioedsraffine} we can make replace Assumption $\mathbf{(BC_1)}$ with Assumptions $\mathbf{(A_1)}$ and $\mathbf{(A_2)}$. We then deduce that BSDE \eqref{chapBioedsraffine} admits a unique solution $(Y,Z)\in \mathcal S_p\times \mathcal H_p$ for any $p>1$ and \eqref{apriori:briandconfortola} holds for any $1<p<p^*$. 
 \end{proof}

\noindent In this particular case, Assumptions $\mathbf{(sD^\infty)}$, $\mathbf{(sH_{1,\infty})}$ and $\mathbf{(sH_{2,\infty})}$ become

\vspace{0.5em}
\noindent $\mathbf{(DA_1)}$ For any $p>1$, $\xi$ belongs to $\mathbb D^{1,p}$ and the stochastic processes $$(t,\omega) \longmapsto \lambda_t(\omega), \mu_t(\omega), \nu_t(\omega)$$ belong to $L^2([0,T]; \mathbb D^{1,p})$.

\vspace{0.5em}
\noindent ($\mathbf{DA_{2}}$) Let $(\varepsilon_k)_{k\in \mathbb N}$ be a sequence in $(0,1]$ such that $\lim\limits_{k\to +\infty} \varepsilon_k=0$, and let  $(Y^k,Z^k)_k$ be a sequence of random variables which converges in $\S_p\times \H_p$ for any $p>1$ to some $(Y,Z)$. Then for all $h\in H$, the following convergences hold in probability
\begin{align*}
&\|\mu_\cdot(\omega+\varepsilon_k h) -\mu_\cdot(\omega)\|_{L^2([0,T])}\underset{k\to+\infty}{\longrightarrow}0, \nonumber\\
&\|\nu_\cdot(\omega+\varepsilon_k h) -\nu_\cdot(\omega)\|_{L^2([0,T])}\underset{k\to+\infty}{\longrightarrow}0.\end{align*}

\begin{Theorem}\label{chapBiothm:diff:affine:BC} Assume that $\mathbf{(A_1)}$, $\mathbf{(A_2)}$, $\mathbf{(DA_1)}$ and $\mathbf{(DA_2)}$ hold. Then, by denoting $(Y,Z)$ the unique solution of \eqref{chapBioedsraffine}, for any $p>1$ and $t\in [0,T]$, $Y_t\in \mathbb D^{1,p}$ and $Z\in L^2([t,T]; \mathbb D^{1,p}) $. Besides, a version of 
$$(D_u Y_t, D_u Z_t)_{0\leq u\leq t, 0\leq t\leq T},$$ is given as the solution to the affine BSDE:
\begin{equation}
\label{chapBioeq:DY:thm:BC:affine}
D_u Y_t = D_u \xi +\int_t^T \left(D_u \lambda_s +D_u\mu_s Y_s+ D_u\nu_s Z_s+ \mu_s D_u Y_s+ \nu_s D_u Z_s\right)ds -\int_t^T D_u Z_s dW_s,
\end{equation}
\end{Theorem}
\begin{proof} Under $\mathbf{(A_1)}$ and $\mathbf{(A_2)}$ and according to Theorem \ref{thm:affine:existence}, BSDE \eqref{chapBioedsraffine} admits a unique solution $(Y,Z)\in \mathcal S_p\times \mathcal H_p$ for any $p>1$. Now, Assumptions $\mathbf{(sD^\infty)}$, $\mathbf{(sH_{1,\infty})}$ and $\mathbf{(sH_{2,\infty})}$ are automatically satisfied under $\mathbf{(DA_1)}$ and $\mathbf{(DA_2)}$. Hence, by applying Theorem \ref{chapBiothm:diff:sl:BC}, we deduce that for any $p>1$ and $t\in [0,T]$, $Y_t\in \mathbb D^{1,p}$ and $Z\in L^2([t,T]; \mathbb D^{1,p}) $ and a version of $(DY_t,DZ_t)$ is given by the solution to BSDE \eqref{chapBioeq:DY:thm:BC:affine}.
\end{proof}

 \section{Existence of densities for the $Y$ component}\label{Mbio:densityY}
 \subsection{The stochastic Lipschitz case}\label{chapBioMbio:section:density:stolip}
We now aim at applying Bouleau and Hirsch's Criterion (see Theorem \ref{BH}) to the $Y$ component of the solution $(Y,Z)$ of BSDE \eqref{chapBioedsrslips}. We set the following assumption

\vspace{0.5em}
\noindent \textbf{(A$^{p,\beta}$)} Let $p$ be in $\in\left(\frac12, 1\right)$, $\beta> \text{max } \{2/(2p-1); 3\}$ and let Assumption $\mathbf{(ELK)^{p,\beta}}$ holds and assume moreover that $\int_0^\cdot \theta_s dW_s\in \text{BMO}(\mathbb P)$, where $\theta$ is defined in Assumption $\mathbf{(sL^{1,\beta})}$.

\vspace{0.5em}
\noindent 
Notice that under Assumption \textbf{(A$^{p,\beta}$)} or Assumption \textbf{(BC)}, we have proved that $Y_t\in \mathbb D^{1,p}$ and $Z\in L^2([0,T];\mathbb D^{1,p})$ for some $p>1$ and that their Malliavin derivatives $(D_rY_t, D_rZ_t)$ satisfy the linear BSDE \eqref{chapBioeq:DY:thm} (see Theorems \ref{chapBiothm:diff:sl} and \ref{chapBiothm:diff:sl:BC}). We then have the following theorem which gives conditions ensuring that, given a time $t$, the law of the first component $Y_t$ of the solution of the non-Markovian BSDE \eqref{chapBioedsrslips} admits a density.
\begin{Theorem}\label{Mbio:THM:density}
Let $\mathbf{(A^{p,\beta})}$ or $\mathbf{(BC)}$ hold. Denote by $(Y,Z)$ the unique solution of BSDE \eqref{chapBioedsrslips}.  If there exists $A\subset\Omega$ such that $\mathbb{P}(A)>0$ and one of the following two assumptions holds for $t\in (0,T]$ and $s\in [t,T]$

\begin{align*} 
(\textsc{sH+})&\hspace{0.8em} D_u \xi \geq  0,\; D_u f(s,Y_s,Z_s) \geq 0,\ \lambda(du)-a.e.,\text{ and }  D_u \xi > 0,\ \lambda(du)-a.e.\text{ on } A\\
&\\
(\textsc{sH-})&\hspace{0.8em}  D_u \xi \leq  0,\; D_u f(s,Y_s,Z_s) \leq 0,\ \lambda(du)-a.e.,\text{ and }  D_u \xi < 0,\ \lambda(du)-a.e.\text{ on } A,
\end{align*}
\noindent then the law of $Y_t$ is absolutely continuous with respect to Lebesgue measure.
\end{Theorem}
\begin{proof}
Under Assumptions \textbf{(A$^{p,\beta}$)} or \textbf{(BC)}, we know from respectively Theorems \ref{chapBioexistence:slipsch} and \ref{chapBiothm:diff:sl} or Theorems \ref{Mbio:thm:apriori:BC} and \ref{chapBiothm:diff:sl:BC}, that BSDE \eqref{Mbio:thm:apriori:BC} admits a unique solution $(Y,Z)$ which is Malliavin differentiable, whose Malliavin derivatives $(D_uY_t,D_uZ_t)_{0\leq u\leq t\leq T}$ are solutions to the following linear BSDE
\begin{align*}
D_u Y_t& = D_u \xi +\int_t^T \left(D_u f(s,Y_s,Z_s) +f_y(s,Y_s,Z_s)D_u Y_s+ f_z(s,Y_s,Z_s) D_u Z_s\right)ds\\
& -\int_t^T D_u Z_s dW_s,\; \forall 0\leq u\leq t\leq T, \; \mathbb P-a.s.
\end{align*}
Notice that for any $(t,y,z)\in [0,T]\times \mathbb R^2$, $|f_z(t,y,z)|\leq \theta_t$ under Assumption \textbf{(A$^{p,\beta}$)}  or $|f_z(t,y,z)|\leq K_t$ under Assumption \textbf{(BC)}. Hence, we can define a probability measure $\mathbb Q$ by
$$\frac{d\mathbb Q}{d\mathbb P}:= \mathcal E\left(\int_0^T f_z(s,Y_s,Z_s)dW_s\right)=e^{\int_0^Tf_z(s,Y_s,Z_s) dW_s-\frac12\int_0^T |f_z(s,Y_s,Z_s)|^2ds},$$ where $ \mathcal E\left(\int_0^\cdot f_z(s,Y_s,Z_s)dW_s\right)$ is a uniformly integrable martingale according to \cite[Theorem 2.3]{Kazamaki}. Changing the Brownian motion according to Girsanov's Theorem and using a linearisation (see \cite{ELKPengQuenez}), we obtain for any $0\leq u\leq t\leq T$
\begin{align*}
D_u Y_t&=\mathbb E^{\mathbb Q}_t \left[D_u \xi e^{\int_t^Tf_y(s,Y_s,Z_s) ds} +\int_t^Te^{\int_t^sf_y(s,Y_s,Z_s)du} D_u f(s,Y_s,Z_s)ds\right]\\
&\geq 0, \; du\otimes d\mathbb P-a.e.
\end{align*}
Moreover, let $A$ be such that $\mathbb P(A)>0$, and $D_u \xi>0$ on $A$. We obtain

\begin{align*}
 D_u Y_t&\geq  \mathbb E_t^\mathbb Q\left[\mathbf 1_A D_u \xi e^{\int_t^T f_y(s,Y_s,Z_s)ds}\right]\\
  &>0.
 \end{align*}
 
\noindent Thus, $\|DY_t \|^2_{L^2([0,T])}>0, \; \mathbb P-a.s.$ and from Theorem \ref{BH} the law of $Y_t$ is absolutely continuous with respect to the Lebesgue measure.\vspace{0.5em}

\noindent The proof under Assumption $($\textsc{sH-}$)$ is similar.
\end{proof}

 \subsection{Non-Markovian Lipschitz BSDEs} \label{chapBioMbio:section:dnesity:lip}
In this section, we study a particular class of stochastic Lipschitz BSDE \eqref{chapBioedsrslips}, which the generator is Lipschitz in its space variables with a nonnegative Lipschitz constant. We provide weaker conditions than Conditions (\textsc{sH+}) and (\textsc{sH-}) ensuring that the law of the component $Y_t$ of the solution to the corresponding Lipschitz BSDE has a density. We consider the following non-Markovian Lipschitz BSDE
\begin{equation}\label{bsde:lipschitz:bio}
Y_t=\xi+\int_t^T f(s,Y_s,Z_s) ds -\int_t^T Z_s dW_s, \; \forall t\in [0,T], \; \mathbb P-a.s.
\end{equation}
 where $\xi$ is an $\mathcal{F}_T$-measurable random variable and $f:[0,T]\times\Omega\times\real^2 \longrightarrow \real$ is an $\F$-progressively measurable process where as usual the $\omega$-dependence is omitted.
We set the following assumption
\begin{itemize}
\item[$\mathbf{(L)}$] 
\begin{itemize}
\item[$(i)$]  The map $(y,z)\longmapsto f(\cdot,y,z)$ is differentiable with continuous partial derivatives uniformly bounded by a positive constant $m$. We denote by $f_y$ (resp. $f_z$) the derivative of $f$ with respect to $y$ (resp. $z$). 
\item[$(ii)$] We have
$$\mathbb E\left[|\xi|^2+\int_0^T |f(s,0,0)|^2 ds\right]<+\infty.$$
\end{itemize}
\end{itemize}

\begin{Theorem}[\cite{ELKPengQuenez}]\label{chapBiolipschitz:existence} Under Assumption $\mathbf{(L)}$, there exists a unique pair of adapted processes $(Y,Z)$ which solves BSDE \eqref{bsde:lipschitz:bio} in $\mathcal S_2 \times \mathcal H_2$. 
\end{Theorem}

\noindent We now turn to the Malliavin differentiability of the solution $(Y,Z)$ of BSDE \eqref{bsde:lipschitz:bio}. This problem was studied in \cite{PardouxPeng_BSDEQuasilinearParabolicPDE} in the Markovian case with Lipschitz coefficients (\textit{i.e.} when the data $\xi,f(t,\cdot,y,z)$ are functions of the solution of a Brownian SDE). It was extended in \cite{ELKPengQuenez} to the non-Markovian case with Lipschitz coefficients. This question was then studied in \cite{GeissSteinicke} for L\'evy driven BSDEs and in \cite{MastroliaPossamaiReveillac_MDBSDE} in which the conditions improve those in \cite{ELKPengQuenez} (see \cite[Section 6.3]{MastroliaPossamaiReveillac_MDBSDE}). In this section, we recall the results of \cite{MastroliaPossamaiReveillac_MDBSDE} where a new criterion ensuring that a random variable is in $\mathbb D^{1,2}$ has been proved. Set the following assumption

\begin{itemize}
\item[$\mathbf{(\textbf{\textsc{lD}})}$] 
\begin{itemize}
\item $\xi\in \mathbb D^{1,2}$, for any $(y,z)\in\mathbb R^2$, $(t,\omega)\longmapsto f(t,\omega,y,z)$ is in $L^2([0,T];\mathbb D^{1,2})$, $f(\cdot,y,z)$ and $Df(\cdot,y,z)$ are $\mathbb F$-progressively measurable, and
$$ \mathbb E\left[\int_0^T \|D_\cdot f(s,Y_s,Z_s)\|_{\mathfrak h}^2 ds\right]<+\infty. $$
\item There exists $p\in(1,2)$ such that for any $h\in H$
$$\hspace{-2em} \lim\limits_{\varepsilon \to 0}\ \E\left[ \left(\int_0^T\left|\frac{f(s,\cdot+\varepsilon h,Y_s,Z_s)-f(s,Y_s,Z_s)}{\varepsilon}-\langle Df(s,Y_s,Z_s),\dot h\rangle_{\mathfrak h}\right| ds\right)^p\right]=0,$$ 
\item Let $(\varepsilon_n)_{n\in \mathbb N}$ be a sequence in $(0,1]$ such that $\lim\limits_{n\to +\infty} \varepsilon_n=0$, and let  $(Y^n,Z^n)_n$ be a sequence of random variables which converges in $\mathcal S_2\times \mathcal H_2$ to some $(Y,Z)$. Then for all $h\in H$, the following convergences hold in probability
\begin{align}
\label{assumpcts1:Lipsch}
&\|f_y(\cdot,\omega+\varepsilon_n h,Y^n_\cdot,Z_\cdot) -f_y(\cdot,\omega,Y_\cdot,Z_\cdot)\|_{\mathfrak h}\underset{n\to+\infty}{\longrightarrow}0 \nonumber\\
& \|f_z(\cdot,\omega+\varepsilon_n h,Y^n_\cdot,Z^n_\cdot)-  f_z(\cdot,\omega,Y_\cdot,Z_\cdot)\|_{\mathfrak h}\underset{n\to+\infty}{\longrightarrow}0,\end{align}
or
\begin{align}
\label{assumpcts2:Lipsch}
\|f_y(\cdot,\omega+\varepsilon_n h,Y^n_\cdot,Z^n_\cdot) -f_y(\cdot,\omega,Y_\cdot,Z_\cdot)\|_{\mathfrak h}\underset{n\to +\infty}{\longrightarrow}0\nonumber\\
 \|f_z(\cdot,\omega+\varepsilon_n h,Y_\cdot,Z^n_\cdot)   -f_z(\cdot,\omega,Y_\cdot,Z_\cdot)\|_{\mathfrak h}\underset{n\to +\infty}{\longrightarrow}0.
\end{align}

\end{itemize}

\end{itemize}

\begin{Theorem}[Theorem 5.1 in \cite{MastroliaPossamaiReveillac_MDBSDE}]\label{Mbio:thm:d12} Let $(Y,Z)$ be the solution of BSDE \eqref{bsde:lipschitz:bio} under Assumption $\mathbf{(L)}$. Let Assumption $\mathbf{(\textbf{\textsc{lD}})}$ be satisfied, then for any $t\in [0,T]$, $Y_t\in \mathbb D^{1,2}$ and $Z\in L^2([t,T]; \mathbb D^{1,2})$. 
\vspace{0.5em}

\noindent Besides, by denoting $DY_t$ (resp. $DZ_t$) the Malliavin derivative of $Y_t$ (resp. $Z_t$), the pair $(DY,DZ)$ satisfies the following (linear) BSDE
\begin{align}\nonumber
D_u Y_t&=D_u \xi +\int_t^T \left(D_u f(s, Y_s,Z_s) + f_y(s,Y_s,Z_s) D_u Y_s +f_z(s,Y_s,Z_s)D_u Z_s\right)ds \\
&\label{edsrderiv:Mbio}-\int_t^T D_uZ_s dW_s, \; 0\leq u \leq t\leq T, \ \mathbb P-a.s.\end{align}

\end{Theorem}
\noindent We now aim at applying Bouleau and Hirsch's Criterion (see Theorem 2.1.2 in \cite{Nualart_Book}) to the $Y$ component of the solution $(Y,Z)$ of BSDE \eqref{bsde:lipschitz:bio}. The existence of a density for $Y_t$ when $t\in (0,T]$ when $f$ is Lipschitz in its space variable was solved in \cite{AntonelliKohatsu} in the Markovian case. We want to extend this result to the non-Markovian case. The following theorem gives conditions which ensure that, given a time $t$, the first component $Y_t$ of the solution of the non-Markovian BSDE \eqref{bsde:lipschitz:bio} admits a density under $(\mathbf{L})$ and $\mathbf{(\textbf{\textsc{lD}})}$. These conditions are similar to those of \cite[Theorem 3.1]{AntonelliKohatsu} in the Lipschitz Markovian case.
Following \cite{AntonelliKohatsu, AbouraBourguin,MastroliaPossamaiReveillac_Density}, let $A$ be a subset of $\Omega$ such that $\mathbb P(A)>0$. We set 
\begin{align*}
 \underline{d\xi}&:=\max \left\{ M\in \mathbb R, \;  D_u \xi \geq M,\  du\otimes \mathbb P-a.e.\right\} ,\\
 \underline{df}(t)&:=\max \left\{  M\in \mathbb R, \; \forall s\in [t,T] \ D_u f(s,Y_s,Z_s) \geq M,\  du\otimes \mathbb P-a.e.\right\},\\
 \underline{d\xi}_A:&=\max \left\{  M\in \mathbb R, \;  D_u\xi \geq M,\ du-a.e. \text{ on } A.\right\},
\end{align*}
\begin{align*}
 \overline{d\xi}&:=\min \left\{  M\in \mathbb R, \; D_u \xi \leq M,\  du\otimes\mathbb P-a.e.\right\} ,\\
\overline{df}(t)&:=\min \left\{  M\in \mathbb R, \; \forall s\in [t,T] \ D_u f(s,Y_s,Z_s) \leq M,\  du\otimes \mathbb P-a.e.\right\} ,\\
\overline{d\xi}_A&:=\min \left\{  M\in \mathbb R, \; D_u \xi \leq M,\  du-a.e. \text{ on } A\right\}.
\end{align*}

\begin{Theorem}\label{chapBiothm:densityb}
Let $(Y,Z)$ be the solution of BSDE \eqref{bsde:lipschitz:bio} under Assumptions $\mathbf{(L)}$ and $\mathbf{(\textbf{\textsc{lD}})}$. Fix some $t\in(0,T]$. If there exists $A\subset\Omega$ such that $\mathbb{P}(A)>0$ and one of the two following assumptions holds

\begin{align*}
 &(\textsc{H+})\quad \begin{cases}
\displaystyle \underline{d\xi}e^{-\text{sgn}(\underline{d\xi})m(T-t)}+\underline{df}(t)\int_t^T e^{-\text{sgn}(\underline{df}(t))m(s-t)}ds\geq0 \\
\displaystyle \underline{d\xi}_Ae^{-\text{sgn}(\underline{d\xi}_A)m(T-t)}+\underline{df}(t)\int_t^T e^{-\text{sgn}(\underline{df}(t))m(s-t)}ds>0
\end{cases}\\[0.3em]
&(\textsc{H-})\quad \begin{cases}
\displaystyle  \overline{d\xi}e^{-\text{sgn}(\overline{d\xi})m(T-t)}+\overline{df}(t)\int_t^T e^{-\text{sgn}(\overline{df}(t))m(s-t)}ds\leq 0 \\
\displaystyle \overline{d\xi}^Ae^{-\text{sgn}(\overline{d\xi}^A)m(T-t)}+\overline{df}(t)\int_t^T e^{-\text{sgn}(\overline{df}(t))m(s-t)}ds<0,
\end{cases}
\end{align*}

then $Y_t$ has a law absolutely continuous with respect to the Lebesgue measure.
\end{Theorem}

\begin{proof} The proof follows the same line than the one of  \cite[Theorem 3.1]{AntonelliKohatsu}. Assume that $($\textsc{H+}$)$ holds. We aim at applying Bouleau-Hirsch Criterion (Theorem 2.1.2 in \cite{Nualart_Book}). From Theorem \ref{Mbio:thm:d12}, $Y_t \in \mathbb D^{1,2}$ and $Z\in L^2([t,T]; \mathbb D^{1,2})$. Let $ 0\leq u\leq t\leq T$, using the linearisation method for BSDE (see  \cite{ELKPengQuenez}) we have
\begin{align*}
D_u Y_t&=D_u \xi +\int_t^T \left(D_uf(s,Y_s,Z_s) + f_y(s,Y_s,Z_s) D_u Y_s +f_z(s,Y_s,Z_s)D_u Z_s\right)ds\\
&\quad -\int_t^T D_u Z_s dW_s, \\
&=\mathbb E_t^\mathbb Q\left[ D_u\xi e^{\int_t^T f_y(s,Y_s,Z_s)ds}+\int_t^T e^{\int_t^s f_y(\alpha,Y_\alpha,Z_\alpha)d\alpha}D_uf(s,Y_s,Z_s)ds\right]\\
&\geq \underline{d\xi} \mathbb E_t^\mathbb Q\left[e^{\int_t^T f_y(s, Y_s,Z_s)ds}\right] + \underline{df(t)} \mathbb E_t^\mathbb Q\left[\int_t^Te^{\int_t^s f_y(\alpha,Y_\alpha,Z_\alpha)d\alpha}ds\right]\\
&\geq \underline{d\xi} e^{-\text{sgn}(\underline{d\xi})m(T-t)} + \underline{df(t)} \int_t^Te^{-\text{sgn}(\underline{df}(s))m(s-t)}ds.  \end{align*}
Hence, $D_u Y_t\geq 0, \; du\otimes \mathbb P-a.e$. Moreover, let $A$ be such that $\mathbb P(A)>0$, we obtain

\begin{align*}
 D_u Y_t&\geq \underline{d\xi}_A \ \mathbb E_t^\mathbb Q\left[\mathbf 1_A e^{\int_t^T f_y(s,Y_s,Z_s)ds}\right] +\underline{df}(t)\mathbb E_t^\mathbb Q\left[\int_t^Te^{\int_t^s f_y(\alpha,Y_\alpha,Z_\alpha)d\alpha}ds\right]\\
 &>0.
 \end{align*}
 
\noindent Thus, $\|DY_t \|^2_{L^2([0,T])}>0, \; \mathbb P-a.s.$ and from Theorem \ref{BH} the law of $Y_t$ is absolutely continuous with respect to the Lebesgue measure.\vspace{0.5em}

\noindent The proof under Assumption $($\textsc{H-}$)$ is similar.
\end{proof}
\subsection{Affine BSDEs with unbounded coefficients}\label{chapBioMbio:section:malliavindiff:affine}
\noindent We study the existence of a density for the first component of the solution to BSDE \eqref{chapBioedsraffine}. We set 
\begin{align*}
 \underline{d\lambda}(t)&:=\max \left\{  M\in \mathbb R, \; \forall s\in [t,T] \ D_u\lambda_s \geq M,\  du\otimes \mathbb P-a.e.\right\} ,\\
\end{align*}
and
\begin{align*}
\overline{d\lambda}(t)&:=\min \left\{  M\in \mathbb R, \; \forall s\in [t,T] \ D_u\lambda_s \leq M,\  du\otimes \mathbb P-a.e.\right\} .\\
\end{align*}
We set also the following assumption
\begin{itemize}
\item[$(\mathcal P+)$] Let $(Y_t,Z_t)$ be the unique solution of BSDE \eqref{chapBioedsraffine} for any $t\in [0,T]$. Then, for any $s\in [0,t]$,
$$D_u \mu_s Y_s+D_u \nu_s Z_s\geq 0, du\otimes \mathbb P-a.s. $$
 \item[$(\mathcal P-)$] Let $(Y_t,Z_t)$ be the unique solution of BSDE \eqref{chapBioedsraffine} for any $t\in [0,T]$. Then, for any $s\in [0,t]$,
$$D_u\mu_s Y_s+D_u\nu_s Z_s\leq 0, du\otimes \mathbb P-a.s. $$
\end{itemize}
\begin{Remark}\label{rem:pb:signZ}
Notice that if $(Y,Z)$ is the unique solution to BSDE \eqref{chapBioedsraffine}, hence using a linearisation method (see \cite{ELKPengQuenez})
$$Y_t:=\mathbb E_t^\mathbb Q\left[ \xi e^{\int_t^T \mu_s ds } +\int_t^T \lambda_s e^{\int_s^t \mu_\alpha d\alpha} ds\right], $$ which is non negative as soon as $\xi$ and $\lambda$ are non negative, where 
$$\frac{d\mathbb Q}{d\mathbb P}:= e^{\int_0^T \nu_sdW_s -\frac12\int_0^T |\nu_s|^2ds},$$ as soon as $\mathcal E\left(\int_0^\cdot \nu_s dW_s\right)$ is a martingale (see Assumption $($\textsc{BMO}$)$ below together with \cite[Theorem 2.3]{Kazamaki}). Thus, as soon as $Y$ is non negative, if $\mu$ is a semi-martingale, we can give conditions which ensures that $D \mu_t$ is non negative using a Lamperti transform (see e.g. \cite[Step 1 of the proof of Theorem 3.3]{AbouraBourguin}).\\

\noindent Concerning the $Z$ process, conditions ensuring that $Z$ is non negative has been obtained in \cite{dosreis_dosreis} in the Markovian case only. In the non-Markovian case, this problem is still open, as far as we know. However, if $\nu$ is deterministic, conditions $(\mathcal P+)$ and $(\mathcal P-)$ can be simplified (see Section \ref{Mbio:section:finance}).

\end{Remark}

\begin{Theorem}\label{chapBiothm:densityaffine} Assume that $\mathbf{(A_1)}, \mathbf{(A_2)}$, $\mathbf{(DA_1)}$ and $\mathbf{(DA_2)}$ hold and let $(Y,Z)$ be the solution of BSDE \eqref{chapBioedsraffine}. If there exists $A\subset\Omega$ such that $\mathbb{P}(A)>0$ and one of the two following assumptions holds

\vspace{0.5em}
\noindent $\mathbf{(\textbf{\textsc{aH}}+)}$ $D_u \xi \geq  0,\ \lambda(du)-a.e.,\;   D_u \xi > 0,\ \lambda(du)-a.e.\text{ on } A,\; \underline{d\lambda}(t) \geq 0 \text{ and } (\mathcal P+) \text{ holds}$,

\vspace{0.5em}
\noindent $\mathbf{(\textbf{\textsc{aH}}-)}$ $D_u \xi \leq  0,\ \lambda(du)-a.e.,\;   D_u \xi < 0,\ \lambda(du)-a.e.\text{ on } A,\; \underline{d\lambda}(t) \leq 0  \text{ and } (\mathcal P-) \text{ holds},$
then $Y_t$ has a law absolutely continuous with respect to the Lebesgue measure.
\end{Theorem}
\begin{proof} We prove the previous theorem under Assumption $\mathbf{(\textbf{\textsc{aH+}})}$. Let $t\in (0,T]$. We know from Theorem \ref{thm:affine:existence} and Theorem \ref{chapBiothm:diff:affine:BC} that BSDE \eqref{chapBioedsraffine} admits a unique solution $(Y_t,Z_t)_{t\in [0,T]}$ such that $Y_t\in \mathbb D^{1,p}$ and $Z\in L^2([t,T]; \mathbb D^{1,p})$ for any $p>1$ with derivatives $(DY_t,DZ_t)$ satisfying the following linear BSDE
\begin{align*}\nonumber
D_u Y_t&=D_u \xi +\int_t^T (D_u \lambda_s+ D_u\mu_s Y_s + D_u \nu_s Z_s+ \mu_s D_u Y_s +\nu_s D_u Z_s) ds \\
&-\int_t^T D_uZ_s dW_s, \; 0\leq u\leq t\leq T, \ \mathbb P-a.s.\end{align*}

\noindent Set $\mathbb Q$ a probability measure defined by
$$\frac{d\mathbb Q}{d\mathbb P}:= \mathcal E\left(\int_0^T \nu_sdW_s\right)=e^{\int_0^T \nu_s dW_s-\frac12\int_0^T |\nu_s|^2ds},$$ where $ \mathcal E\left(\int_0^\cdot \nu_sdW_s\right)$ is a uniformly martingale according to \cite[Theorem 2.3]{Kazamaki}. Changing the Brownian motion according to Girsanov's Theorem and using a linearisation, we obtain for any $0\leq u\leq t\leq T$

\begin{equation}\label{dY_affine}
D_u Y_t=\mathbb E^{\mathbb Q}_t \left[D_u \xi e^{\int_t^T \mu_s ds} +\int_t^T (D_u \lambda_s+ D_u\mu_s Y_s + D_u \nu_s Z_s)e^{\int_t^s \mu_\alpha d\alpha}ds\right].
\end{equation}
Thus, by reproducing the proof of Theorem \ref{chapBiothm:densityb}, we show that $\|DY_t \|_{L^2([0,T])}>0, \; \mathbb P-a.s.$ and from Theorem \ref{BH} the law of $Y_t$ is absolutely continuous with respect to the Lebesgue measure.\\

\noindent The proof under $\mathbf{(\textbf{\textsc{aH-}})}$ is similar. 
\end{proof}
\begin{Remark} \label{chapBioremarksigne}Under the same assumptions than in the previous theorem and assuming that $($\textsc{aD+}$)$ holds (resp. $($\textsc{aD-}$)$ holds), the proof shows that in fact $D_uY_t\geq 0, du\otimes\mathbb P-a.e.$ (resp. $D_uY_t\leq 0, du\otimes\mathbb P-a.e.$). 
\end{Remark}
\section{Existence of a density for the $Z$-component: a still open problem in the non-Markovian case}
\label{chapBioMbio:section:discussion}
In this section we turn to conditions ensuring the existence of densities for the laws of $Z_t$ components of solutions to stochastic Lipschitz BSDEs. We begin to investigate this problem for a particular class of stochastic Lipschitz BSDE with a linear generator with respect to the $z$ component by following the same proofs that in \cite{AbouraBourguin} and we explain why we are not able to extend results obtained in \cite{MastroliaPossamaiReveillac_Density} to the non-Markovian case for general stochastic Lipschitz BSDEs.

 \subsection{A result for BSDE with linear generator with respect to $z$}\label{soussection:z:linear}
Consider the following BSDE

 \begin{equation}\label{chapBioedsr:linearbsde}
Y_t=\xi+\int_t^T(\tilde{f}(s,Y_s)+\theta_s Z_s)ds -\int_t^T Z_s dW_s, \; \forall t\in [0,T], \; \mathbb P-a.s.
\end{equation}

\noindent where $\theta$ is a square integrable adapted process. In this case, recall that under $\mathbf{(EKH)^{p,\beta}}$ for any $p$ be in $\in\left(\frac12, 1\right)$, $\beta> \text{max } \{2/(2p-1); 3\}$ or $\mathbf{(BC)}$, according to Theorem \ref{chapBiothm:diff:sl} or respectively Theorem \ref{chapBiothm:diff:sl:BC}, BSDE \eqref{chapBioedsr:linearbsde} admits a unique solution in $\mathcal S_p\times \mathcal H_p$ and $t\in [0,T]$, $Y_t\in \mathbb D^{1,2}$ and $Z\in L^2([t,T]; \mathbb D^{1,2}) $. Besides, a version of 
$$(D_u Y_t, D_u Z_t)_{0\leq u\leq t, 0\leq t\leq T}$$ is given by the solution to the affine BSDE:

\begin{equation}\label{edsr:linear:densityz} D_u Y_t = D_u\xi +\int_t^T \left(D_u \tilde f(s,Y_s)+f_y(s,Y_s)D_uY_s+\theta_sD_u Z_s\right)ds -\int_t^T D_u Z_s dW_s.
\end{equation}
As explain in Remark \ref{rem:pb:signZ}, in order to obtain a sign for the Malliavin derivative of the $Y$ component of the solution to an affine BSDE with unbounded coefficients when we have no information on the sign of the $Z$ process, we must assume that $\theta$ is deterministic to apply Theorem \ref{chapBiothm:densityaffine}. Thus, we set the following assumption
\begin{itemize}
\item[$\mathbf{(\Theta)}$] The process $\theta$ defined in BSDE \eqref{chapBioedsr:linearbsde} is deterministic.
\end{itemize}

\noindent Let now $Y$ be the first component of the solution to BSDE \eqref{edsr:linear:densityz}. We set for any $0\leq v\leq t\leq T$

\vspace{0.5em}
\noindent $\mathbf{(DY+)}$ $D_v \xi\geq 0, \; D_{v} \tilde f(t,Y_t)\geq 0,$ $\mathbb P-$a.s.,

\vspace{0.5em}
\noindent 
$\mathbf{(DY-)}$   $D_v \xi\leq 0, \; D_{v} \tilde f(t,Y_t)\leq 0, $ $\mathbb P-$a.s.

 \begin{Remark}
Similarly to Remark \ref{chapBioremarksigne}, Notice that the proof of Theorem \ref{Mbio:THM:density} shows that for any $0\leq v\leq s\leq T$:
\begin{itemize}
\item[] Under Assumption $\mathbf{(DY+)}$
\begin{equation}\label{densityZ:signeDY}
D_v Y_t \geq 0, \; \mathbb P-a.s.
\end{equation}
\item[] Under Assumption $\mathbf{(DY-)}$
\begin{equation}\label{densityZ:signeDY-}
D_v Y_t \leq 0, \; \mathbb P-a.s.
\end{equation}
\end{itemize}

 \end{Remark}
\noindent We have the following theorem which provide conditions on the data $\xi$ and $\tilde f$ ensuring that the law of $Z_t$ has a density with respect to the Lebesgue measure, which can be seen as an extension of \cite[Theorem 4.3]{AbouraBourguin} to the stochastic Lipschitz case
\begin{Theorem}\label{densityZ:linear} Let $\mathbf{(\Theta)}$ be hold and let $(Y,Z)$ be the unique solution of BSDE \eqref{chapBioedsr:linearbsde}. Let $\xi$ be in $\mathbb D^{2,2}$, assume moreover that $\tilde f$ is twice continuously differentiable with respect to $y$. We set the following assumptions for any $0\leq t,t'\leq T$:

\begin{itemize}
\item[$(f+)$]  $D_t \tilde f_y(t,Y_t),\;  (D_t \tilde f)_y(t,Y_t)\geq 0,$
\item[$(f-)$]   $D_t \tilde f_y(t,Y_t),\;  (D_t \tilde f)_y(t,Y_t)\leq 0.$
\end{itemize}

\noindent Assume that there exists $A$ such that $\mathbb P(A)>0$ such that one of the following assumptions is satisfied

\vspace{0.5em}
\noindent $\mathbf{(DZ+)}$
\begin{align}\nonumber &D_{t'}(D_t\xi)\geq 0,\ \mathbb P-a.e. \text{, } D_{t'}(D_t\xi)>0 \text{ on A,}\\
\label{edsr:linear:z:f}& D_{t'}(D_t \tilde f)(t,Y_t)\geq 0
\end{align}
and Assumptions $\mathbf{(DY+)}$ and $(f+)$ hold, or Assumptions $\mathbf{(DY-)}$ and $(f-)$ hold,

\vspace{0.5em}
\noindent $\mathbf{(DZ-)}$ 
\begin{align}\nonumber &D_{t'}(D_t\xi)\leq 0,\ \mathbb P-a.e. \text{, } D_{t'}(D_t\xi)<0 \text{ on A,}\\
\label{edsr:linear:z:fbis}&D_{t'}(D_t \tilde f)(t,Y_t)\leq 0
\end{align}

\noindent and Assumptions $\mathbf{(DY-)}$ and $(f+)$ hold, or Assumptions $\mathbf{(DY+)}$ and $(f-)$ hold.
\vspace{0.5em}

\noindent Then, the law of $Z_t$ is absolutely continuous with respect to Lebesgue's measure for any $t\in(0,T]$.
\end{Theorem}

\begin{proof}
Let $(Y,Z)$ be the unique solution to BSDE \eqref{edsr:linear:densityz} in $\mathbb D^{1,2}\times L^2([0,T];\mathbb D^{1,2})$, which the Malliavin derivatives are solutions to BSDE \eqref{edsr:linear:densityz}.\vspace{0.3em}

\noindent Let Assumption $\mathbf{(DZ+)}$ be true together with Assumption $\mathbf{(DY+)}$ and $(f+)$. We follow the proof of \cite[Theorem 4.3]{AbouraBourguin} by taking the advantage of the representation of the $Z$ process with Clark-Ocone Formula.   Using now a linearization and according to Clark-Ocone Formula, we obtain
$$Z_t=\mathbb E_t^{\mathbb Q}\left[D_t \xi e^{\int_t^T \tilde f_y(s,Y_s)ds}+\int_t^T D_t \tilde f(s,Y_s) e^{\int_t^s \tilde f_y(u,Y_u)du} ds\right], $$
with $\frac{d\mathbb Q}{d\mathbb P}= \exp\left(\int_0^T \theta_s dW_s-\frac12\int_0^T |\theta_s|^2 ds \right)$. Let $0\leq v\leq t$, we have
\begin{align}
\nonumber D_v Z_t&=\mathbb E_t^{\mathbb Q}\left[D_v(D_t \xi)e^{\int_t^T \tilde f_y(s,Y_s)ds}+ D_t \xi e^{\int_t^T \tilde f_y(s,Y_s)ds} \int_t^T D_v  \tilde f_y(s,Y_s)ds \right.\\
\label{d2Z:CO}&\left. +\int_t^T e^{\int_t^s \tilde f_y(u,Y_u)du} \left(D^2_{v,t} \tilde f(s,Y_s) + D_t \tilde f(s,Y_s)\int_t^s D_v \tilde f_y(u,Y_u)du\right) ds\right]. \end{align}

\noindent Hence, using the definition \eqref{D2:definition:symetrie} of $D^2_{v,t} \tilde f$, Inequality \eqref{densityZ:signeDY}, Assumption $(f+)$ and Assumption \eqref{edsr:linear:z:f}, we deduce that for any $0\leq v<t\leq T$, $D_v Z_t >0$. Thus, the law of $Z_t$ has a density for any $t\in(0,T]$ as a consequence of Theorem \ref{BH}.\vspace{0.3em}

\noindent The proof under Assumptions $\mathbf{(DZ+)}$, $\mathbf{(DY-)}$ and $(f-)$ is similar, by using \eqref{d2Z:CO}, Inequality \eqref{densityZ:signeDY-}, Assumption $(ii)$ and Assumption \eqref{edsr:linear:z:f}.\vspace{0.3em}

\noindent Concerning Assumption $\mathbf{(DZ-)}$ we follow exactly the same proof and for any $0\leq v\leq t\leq T$, we show that $D_v Z_t <0,\; \mathbb P-a.s.$.\end{proof}

\begin{Remark}
Theorem \ref{densityZ:linear} extends the results in \cite{AbouraBourguin}. In the present paper $\theta$ is assumed to be a deterministic map behind the $z$ part of the generator, unlike the model studied in \cite{AbouraBourguin} in which the coefficient behind $z$ is constant. Moreover, in our model $\tilde f$ is stochastic Lipschitz with respect to its $y$ variable, whereas it is assumed to be Lipschitz in \cite{AbouraBourguin}. Finally, we deal with the non-Markovian case for both the terminal condition and the generator of the BSDE, whereas \cite{AbouraBourguin} considers the case where only the terminal condition is non-Markovian. 
\end{Remark}
\subsection{Some remarks on the general stochastic Lipschitz case}\label{someremark:z}
Existence of density for the $Z$ component has been studied for quadratic growth BSDEs in \cite{MastroliaPossamaiReveillac_Density} in the Markovian case. We can in fact adapt this proof to the Markovian stochastic Lipschitz case and one could show that conditions ensuring that the law of $Z_t$ component has a density are similar to those obtained for Markovian quadratic growth BSDE (see \cite[Section 4.3]{MastroliaPossamaiReveillac_Density}). Although in the latter paper, the authors obtain conditions which ensure that $Z_t$ admits a density, we can not reproduce the proof here since it is essentially based on Ma-Zhang Representation (see \cite[Lemma 2.4]{MaZhang_PathRegularity}) which holds in the Markovian case. More precisely, we consider the following forward-backward SDE
\begin{equation}\label{chapBioMbio:fbsde}
\begin{cases}
\displaystyle X_t= X_0+\int_0^t b(s,X_s) ds +\int_0^t \sigma(s,X_s)dW_s,\ t\in[0,T],\ \P-a.s.\\
\displaystyle Y_t = g(X_T) +\int_t^T f (s,X_s,Y_s,Z_s) ds -\int_t^T Z_s dW_s, \ t\in [0,T],\ \P-a.s.
\end{cases} 
\end{equation}  
Then, under some conditions on the data of such forward-backward system, denoting by $(X,Y,Z)$ the solution of \eqref{chapBioMbio:fbsde}, there exists a version of $(D_u X_t, D_u Y_t, D_u Z_t)$ for all $0<u\leq t\leq T$ which satisfies:
\begin{align*}
D_u X_t&=\nabla X_t (\nabla X_u)^{-1}\sigma(u,X_u),\\
 D_u Y_t&=\nabla Y_t (\nabla X_u)^{-1}\sigma(u,X_u),\\
 D_u Z_t&=\nabla Z_t (\nabla X_u)^{-1}\sigma(u,X_u),
\end{align*}
where $(\nabla X,\nabla Y,\nabla Z)$ is the solution to the following FBSDE:
\begin{equation}\label{chapBioedsr_gradient}
\begin{cases}
\displaystyle \nabla X_t= \int_0^t b_x(s,X_s) \nabla X_s ds +\int_0^t \sigma_x(s,X_s)\nabla X_sdW_s,\\
\displaystyle \nabla Y_t = g'(X_T)\nabla X_T +\int_t^T\left( f_x(s,X_s,Y_s,Z_s)\nabla X_s + f_y(s,X_s,Y_s,Z_s)\nabla Y_s\right.\\
\displaystyle \hspace{11em}\left. + f_z(s,X_s,Y_s,Z_s)\nabla Z _s \right)ds-\int_t^T \nabla Z_s dW_s.
\end{cases} 
\end{equation} 
\noindent As far as we know, the same kind of decomposition is still open for path-dependent BSDEs. However, it seems to be hard to obtain a similar formula in the path-dependent framework. As an example, let $Y_T=\xi=\int_0^T B_s ds$. Hence, $Y_T\in \mathbb D^{1,2}$ and $D_rY_T= T-r$. In order to separate the Malliavin integration variable $r$ and the time variable $T$ as in \cite[Lemma 2.4]{MaZhang_PathRegularity} for Markovian BSDEs, we could similarly compute the gradient in space of $\xi$ using a Fr\'echet derivative, denoted by $\nabla_F \xi$. Let $x$ be in $\mathcal C([0,T];\mathbb R)$, that is the space of $\mathbb R$-valued continuous functions of $[0,T]$, and set for any $0\leq t,s \leq T$
$$ B_s^{t,x}:= x(s) \mathbf{1}_{s\leq t}+(x_t+B_s-B_t) \mathbf{1}_{s\geq t},$$ where $x(s)$ denotes the path of $x$ up to time $s$. Then, $\nabla_F B_s=1$ for any $s\in [0,T]$. We thus obtain $\nabla_F \xi=\int_0^T \nabla_F B_s ds=T$. The relation between $\nabla_F \xi$ and $D_r \xi$ is not clear and we can not hope to obtain a decomposition as \cite[Lemma 2.4]{MaZhang_PathRegularity} for path-dependent BSDEs using the same method. 
\vspace{0.5em}

\noindent An other approach to study the $Z$ component could consist in studying the path-dependent PDE associated with the path-dependent BSDE, see \textit{e.g.} \cite{PengWang, EkrenKellerTouziZhang, ETZ1, RenTouziZhang1}. Indeed, it is proved, in the latter papers, that the $Z$ component of the solution to a path-dependent BSDE can be expressed through the Dupire derivative of the solution to a path-dependent PDE. It will be then interesting to take advantage of this relation together with the lifting theorem \cite[Theorem 6.1]{ContFournie} to study the $Z$ component. 

\vspace{0.5em}

\noindent Notice nevertheless that in the biological example proposed in Section \ref{chapBio:section:gene}, only the existence of a density for the law of the $Y$ component is relevant to validate the proposed model. In the examples in Finance proposed in Section \ref{Mbio:section:finance}, the model of pricing studied will be reduced to solve BSDE \eqref{chapBioedsr:linearbsde}, hence we will prove that both the law of $Y_t$ and the law of $Z_t$ have densities with respect to the Lebesgue measure.

\section{Application to the gene expression modelling}\label{chapBio:section:gene}

\subsection{Stochastic model of gene expression}
\noindent Stochastic models predicting mRNA and
proteins fluctuations were introduced during the 70's (see \textit{e.g.} \cite{rigney_schieve}). It has become
during this last decades a prolific field in the studying of proteins synthesis known as the "gene expression noise". This section being a mathematical study of a biological problem, we consider one active gene which synthesises one protein and we give a very simplified explanation of the proteins degradation proceed, by focusing on the main step of the mechanism. For more details, see for instance \cite{paulsson}.
\vspace{0.3em}

\noindent \textbf{Step 1: Transcription.} The first step of the synthesis of the protein consists in the transcription of a gene, made of a piece of DNA, into mRNA. The synthesis of mRNA is catalysed by an enzyme, the RNA polymerase whose the activation rate is denoted by $R$.
\vspace{0.3em}

\noindent \textbf{Step 2: Translation.} In this step, the mRNA, previously synthesised, is decoded by a ribosome. A transfer RNA brings amino acids to the ribosome to produce an amino acid chain using the genetic code. The degradation rate of mRNA is denoted by $\rho$. At the end of this step, the protein is synthesised.
\vspace{0.3em}

\noindent Here, we assume that the present protein concentration is known and we want to study the previous protein concentrations which lead to the one observed. As an illustration of this phenomenon, we consider for instance a necrotic cells model, in which we want to control the initial protein concentration. It was showed in \cite{shamarova_etal.} that this
problem can be reduced to solve the following BSDE
\begin{equation}\label{bsdegene}
Y_t=\xi+\int_t^T( f(Y_s)-\rho Y_s) ds -\int_t^T Z_s dW_s,
\end{equation}
where $Y_t$ is the protein concentration at time $t$, $\xi$ is the terminal protein concentration, which is typically the observed data in a necrotic model, and $f$ is the degradation/syntetization rate of the protein depending on $R$,$\rho$ and a positive constant $a$. In this study, following \cite{shamarova_etal.} we assume that $f$ is the Hill function of the protein with coefficient $2$, \textit{i.e.}

$$ f(Y_s):= R\frac{aY_s^2}{1+aY_s^2}.$$
In biochemistry, $f$ quantifies the fraction of the ligand-binding sites on the receptor protein. The Hill coefficient is $2$, and describes cooperativeness effects. In order to validate their model, the authors of \cite{shamarova_etal.} need to compare the law of the protein concentration at time $t$ obtained by solving BSDE \eqref{bsdegene} with the data produced by Gillespie Method (see \cite{Gillespie77}). However, in \cite{shamarova_etal.}, the authors assumed implicitly that $Y_t$ admits a density. 
\vspace{0.3em}

\noindent We propose in this section to apply the results of Section \ref{chapBioMbio:section:dnesity:lip} to study mathematically the model proposed in \cite{shamarova_etal.} when $\xi:= c+W_T$, with the Malliavin calculus. It can be seen as a mathematical strengthening of the model developed in \cite{shamarova_etal.} by using Nourdin and Viens' Formula to obtain estimates of the density.
\begin{Proposition}\label{Mbio:propsigne}
Let $(Y,Z)$ be the unique solution of BSDE \eqref{bsdegene}. Assume that $\xi\geq 0, \; \mathbb P-a.s.$  (resp. $\xi\leq 0, \; \mathbb P-a.s.$), then, for any $t\in [0,T]$, $Y_t\geq 0, \; \mathbb P-a.s.$ (resp. $Y_t\leq 0, \; \mathbb P-a.s.$).
\end{Proposition}
\begin{proof}
We reproduce here the linearisation method for BSDE introduced in \cite{ELKPengQuenez} for BSDE \eqref{bsdegene}.
$$Y_t=\xi+\int_t^T   \left(R\dfrac{aY_s}{1+ aY_s^2} -\rho\right) Y_s ds-\int_t^T Z_s dW_s,$$
hence, by setting $X_t:=Y_t e^{\int_0^t \left(R\frac{aY_s}{1+ aY_s^2} -\rho\right) ds}$, we obtain from Ito's Formula,
\begin{align*}
dX_t&= dY_t e^{\int_0^t \left(R\frac{aY_s}{1+ aY_s^2} -\rho \right)ds} + Y_t e^{\int_0^t \left(R\frac{aY_s}{1+ aY_s^2} -\rho \right)ds}\left(R\dfrac{aY_t}{1+ aY_t^2} -\rho\right)dt\\
&=-Y_te^{\int_0^t \left(R\frac{aY_s}{1+ aY_s^2} -\rho\right) ds}\left(R\dfrac{aY_t}{1+ aY_t^2} -\rho\right)dt+ Z_t e^{\int_0^t \left(R\frac{aY_s}{1+ aY_s^2} -\rho\right) ds}dW_t \\
&+ Y_t e^{\int_0^t \left(R\frac{aY_s}{1+ aY_s^2} -\rho\right) ds}\left(R\dfrac{aY_t}{1+ aY_t^2} -\rho\right)dt\\
\end{align*}
Thus,
$$Y_t=\mathbb E_t\left[\xi e^{\int_t^T \left( R\frac{aY_s}{1+ aY_s^2} -\rho\right) ds}\right], $$
whose sign is fully determined by the sign of $\xi$.

\end{proof}

\subsection{A model which guarantees Gaussian tails.}\label{section:model:gaussiantails}
We extend in this section the model introduced in \cite{shamarova_etal.}. We assume that $R, \rho$ are two real constants  and that $\xi$ satisfied the following assumption
\begin{itemize}
\item $\xi$ is a Gaussian $\mathcal F_T$-measurable random variable whose mean is denoted by $c$ and variance is denoted by $\sigma^2$.
\item $\xi$ is in $\mathbb D^{1,2}$ and there exist $0<\underline{k}\leq \overline{k}$ such that for any $r\in [0,T]$, $0<\underline{k}\leq D_r \xi\leq \overline{k}$.
\end{itemize}
According to Theorem \ref{chapBiolipschitz:existence} and Theorem \ref{Mbio:thm:d12} above, BSDE \eqref{bsdegene} admits a unique solution $(Y,Z)$ such that for any $t\in [0,T]$, $Y_t\in \mathbb D^{1,2}$ and $Z\in L^2([t,T];\mathbb D^{1,2})$. We then have the following proposition.
\begin{Proposition}\label{estimesra}
The first component $Y$ of the solution of BSDE \eqref{bsdegene} admits a density denoted by $\rho_{Y_t}$ at any time $t\in (0,T]$. Besides, $\rho_{Y_t}$ has Gaussian estimates, satisfying the following inequalities for any $x\in \mathbb R$ 

\begin{equation}\label{Mbio:densityestimates}f_i(x)\leq \rho_{Y_t}(x) \leq f_s(x),\end{equation}
where
$$f_i(x)=\frac{C_{Y_t}}{\overline k^2 t}e^{-2\overline{C_{a,R,\rho}}(T-t)} e^{-e^{-2\underline{C_{a,R,\rho}}(T-t)}\frac{(x-\mathbb E[Y_t])^2}{2\underline k^2 t}}, $$
$$f_s(x)= \frac{C_{Y_t}}{\underline k^2 t}e^{-2\underline{C_{a,R,\rho}}(T-t)} e^{-e^{-2\overline{C_{a,R,\rho}}(T-t)}\frac{(x-\mathbb E[Y_t])^2}{2\overline k^2t}},$$
and with 
$$C_{Y_t}:=  \frac{\mathbb E[|Y_t-\mathbb E[Y_t]|]}{2},$$
$$\overline{C_{a,R,\rho}}:= \frac{9}{8}R\sqrt{\frac a3}-\rho,$$
$$\underline{C_{a,R,\rho}}:=- \frac{9}{8}R\sqrt{\frac a3}-\rho.$$
\end{Proposition}

\begin{proof} Let $(Y,Z)$ be the unique solution of \eqref{bsdegene}. We deduce from Theorem \ref{chapBiothm:densityb} that for any $t\in (0,T]$, the law of $Y_t$ admits a density denoted by $\rho_{Y_t}$. Recall that $(DY, DZ)$ satisfies the following linear BSDE
$$D_u Y_t=D_r \xi+\int_t^T 2R \dfrac{aY_sD_uY_s}{(1+aY_s^2)^2}-\rho D_u Y_s \ ds-\int_t^T D_uZ_s dW_s, \; 0\leq u\leq t\leq T, \; \mathbb P-a.s. $$
By linearisation, we thus obtain
$$D_u Y_t=\mathbb E_t\left[ D_u \xi e^{\int_t^T \left(2R\frac{aY_s}{(1+aY_s^2)^2}-\rho\right) ds}\right].$$
Notice that $\overline{C_{a,R,\rho}}:=\frac{9}{8}R\sqrt{\frac a3}-\rho$ is the maximum of $y\longmapsto 2R\frac{ay}{(1+ay^2)^2}-\rho$ and $\underline{C_{a,R,\rho}}:=-\frac{9}{8}R\sqrt{\frac a3}-\rho$ is its minimum. Hence, for any $0\leq u\leq t\leq T$, 

$$\underline k e^{\underline{C_{a,R,\rho}}(T-t)} \leq D_uY_t \leq \overline k e^{\overline{C_{a,R,\rho}}(T-t)}.$$
Using the definition \eqref{def:gF} of $g_{Y_t}$, one get for any $t\in (0,T]$
$$|\underline k|^2te^{2\underline{C_{a,R,\rho}}(T-t)} \leq g_{Y_t}(x) \leq |\overline k|^2te^{2\overline{C_{a,R,\rho}}(T-t)}, \; x\in \mathbb R. $$
Thus, according to Theorem \ref{thm_NourdinViens}, Relation \eqref{Mbio:densityestimates} holds.
\end{proof}

\subsection{Example 1: Shamarova-Ramos-Aguiar's Model}
To validate the method proposed in \cite{shamarova_etal.}, we have to analyse how close the law of $Y_t$ for any $t\in (0,T]$ is to Gaussian distributions produced by the Gillepsie method (see \cite[Section III]{shamarova_etal.}). Notice that in \cite{shamarova_etal.} the law of $Y_t$ is emphasised through a distribution fitting and is not proved rigorously. We propose in this section a more accurate proof in order to validate the Shamarova-Ramos-Aguiar model.

\subsubsection{Law of $Y_t$ by using statistical tests}
 Let $(Y,Z)$ be the unique solution to BSDE \eqref{bsdegene} where $\xi$ has a normal distribution. In \cite{shamarova_etal.}, the authors study their model by assuming that $Y_t$ has a normal distribution and compare the first and second order moments of $Y_t$ with those generated by a benchmark random variable, which has a normal distribution. However, it is not clear that the law of $Y_t$ is normal. Nevertheless, from a statistical point of view, we could validate this assumption by using a statistical hypothesis test. In this subsection, we set the statistical hypothesis
 
 \begin{center}(H) "$Y_t$ has a normal distribution"\end{center} and we first test it using a Jarque-Bera test with the data of \cite[A. Self-regulating gene]{shamarova_etal.}.\vspace{0.3em}
 
 \noindent Recall that the Jarque-Bera test consists in computing the sample skewness, denoting by $S$, and the sample kurtosis, denoting by $K$, of a sample data, such that
 \begin{align*}
 S:= \frac{\frac1M \sum_{i=1}^M (Y_t^i- \overline{Y}_t)^3}{\left(\frac1M \sum_{i=1}^M |Y_t^i- \overline{Y}_t|^2 \right)^{\frac32}},\;  K:=  \frac{\frac1M \sum_{i=1}^M (Y_t^i- \overline{Y}_t)^4}{\left(\frac1M \sum_{i=1}^M |Y_t^i- \overline{Y}_t|^2 \right)^{2}},
 \end{align*}
 where $M$ denotes the size of the sample, $Y_t^i$ is the $ith$ observated data and $ \overline{Y}_t$ is the arithmetical mean of the data. We then define the Jarque-Bera variable denoted by $JB$, by the following formula
 $$JB:= M\left( \frac{S^2}{6}+\frac{1}{24} (K-3)^2\right). $$
Under $(H)$, the law of $JB$ is a chi-squared distribution with two degrees of freedom. Hence, by choosing a risk level $\alpha=5\%$, the critical region is $JB>5.99$, that is to say if $JB>5.99$ we reject $(H)$. We refer to \cite{JB} for more details on this method.\vspace{0.3em}

\noindent We apply this test to $Y_t$, with the data of \cite[A. Self-regulating gene]{shamarova_etal.}: $M=5 000, \; R=1,\; \rho=0.001, \; T=400$. The results are given in Table \ref{table_JB}. 
 
 \begin{table}[H]\caption{ \label{table_JB} A Jarque-Bera test for Hypothesis $(H)$ with the data of  \cite[A. Self-regulating gene]{shamarova_etal.}.} 
 \begin{center}
\begin{tabular}{|c|c|c|c|c|}\hline
Time $t$&$JB$&$(H)$\\
\hline $400$&2.62& Not rejected\\
\hline $300$&7.92& Rejected\\
\hline $200$&5.52& Rejected\\
\hline $100$&19.4& Rejected\\
\hline $50$ &11.45& Rejected\\
\hline
\end{tabular}
\end{center}
 \end{table}
 \paragraph{Interpretation} A Jarque-Bera test does not accept the assumption $(H)$ with a risk level $\alpha= 0.05$. Hence, from a statistical point of view, it is not clear that $Y_t$ has a gaussian law. The problem comes from the number of simulations which has to be high.\vspace{0.5em}

\noindent We now choose a number of simulation more relevant, by taking $M=100000$. We use a Jarque-Bera test together with a Kolmogorov-Smirnov test in order to validate satistically the model developed in \cite{shamarova_etal.}. Recall that if we have a sample $(Y_t^i)_{1\leq i\leq M}$ of observed data, we set $KS$ the Kolmogorov-Smirnov statistic corresponding to the sample, defined by
$$KS:= \sqrt{M} \sup_{x}\left\{F_M(x)-F(x) \right\}, $$
where $F_M$ is the empirical distribution function of the sample of observed data and $F$ is the cumulative distribution function of a normal law with parameters the arithmetic mean and the variance of the sample. Hence, for a level $\alpha=0.05$, by using a Kolmogorov-Smirnov test, we reject the Hypothesis $(H)$ as soon as $KS>1.36$. The results are presented in Table \ref{JB:100000}.
 
  \begin{table}[H]\caption{ \label{JB:100000} A Jarque-Bera test and a Kolmogorov-Smirnov test for Hypothesis $(H)$ with the data of  \cite[A. Self-regulating gene]{shamarova_etal.} and $M=100000$. We write "Not R." for "not rejected".} 
 \begin{center}
\begin{tabular}{|c|c|c|c|c|c|}\hline  \diagbox{Statistical tests}{Time t}&400&300&200&100&50\\
\hline Jarque-Bera test &&&&&\\
$JB$ &1.91&2.61&2.08&2.31&1.72\\
$(H)$& Not R. &Not R.& Not R.&Not R.&Not R. \\
\hline Kolmogorov-Smirnoff test&&&&&\\
$KS$ &0.501&0.500&0.501&0.501&0.501 \\
$(H)$&Not R.&Not R.&Not R. &Not R.& Not R.\\
\hline
\end{tabular}
\end{center}
 \end{table}
  \paragraph{Interpretation} A Jarque-Bera test together with a Kolmogorov-Smirnov test cannot invalidate the assumption $(H)$ with a risk level $\alpha= 0.05$. Hence, from a statistical point of view, the model developed in \cite{shamarova_etal.} seems to be relevant. However, we propose in the next section a pure mathematical analyse of this model, by using the Malliavin calculus and by applying results of \cite{MastroliaPossamaiReveillac_Density} together with those obtained in Section \ref{Mbio:densityY}.

 \subsubsection{Validation of the model by using the Malliavin calculus and Nourdin-Viens Formula}
\noindent Assume that $\xi=c+\sigma^2 W_T$. Then, we can use the result of Section \ref{section:model:gaussiantails} and we deduce that BSDE \eqref{bsdegene} admits a unique solution $(Y,Z)$ such that for any $t\in [0,T]$, $Y_t\in \mathbb D^{1,2}$ and $Z\in L^2([t,T];\mathbb D^{1,2})$. Besides, according to Proposition \ref{estimesra}, for any $t\in [0,T]$, $Y_t$ admits a density with respect to the Lebesgue measure denoted by $\rho_{Y_t}$. such that $\rho_{Y_t}$ has Gaussian estimates, satisfying the following inequalities for any $x\in \mathbb R$ 
\begin{equation*}f_i(x)\leq \rho_{Y_t}(x) \leq f_s(x),\end{equation*}
where
$$f_i(x)=\frac{C_{Y_t}}{ \sigma^4t}e^{-2\overline{C_{a,R,\rho}}(T-t)} e^{-e^{-2\underline{C_{a,R,\rho}}(T-t)}\frac{(x-\mathbb E[Y_t])^2}{2t\sigma^4}},$$
$$ f_s(x)= \frac{C_{Y_t}}{ \sigma^4t}e^{-2\underline{C_{a,R,\rho}}(T-t)} e^{-e^{-2\overline{C_{a,R,\rho}}(T-t)}\frac{(x-\mathbb E[Y_t])^2}{2t\sigma^4}}.$$
We illustrate these results in Figure \ref{figuredensity}.

\begin{figure}[H]\caption{ \label{figuredensity} $T,a,c,\sigma^2=1$, $R=1$,$\rho=0.001$, and 500 000 simulations using a method of Monte-Carlo (see \cite{BouchardTouzi} for instance) to compute the solution of BSDE \eqref{bsdegene}. We represent $\rho_{Y_t}$ for $t=0.9,0.75,0.6,0.5$. We provide in red (resp. in blue) the supremum bound "fs" of $\rho$ (resp. the infimum "fi"), using Nourdin and Viens' Formula.}
 \begin{center}
 \includegraphics[scale=0.25]{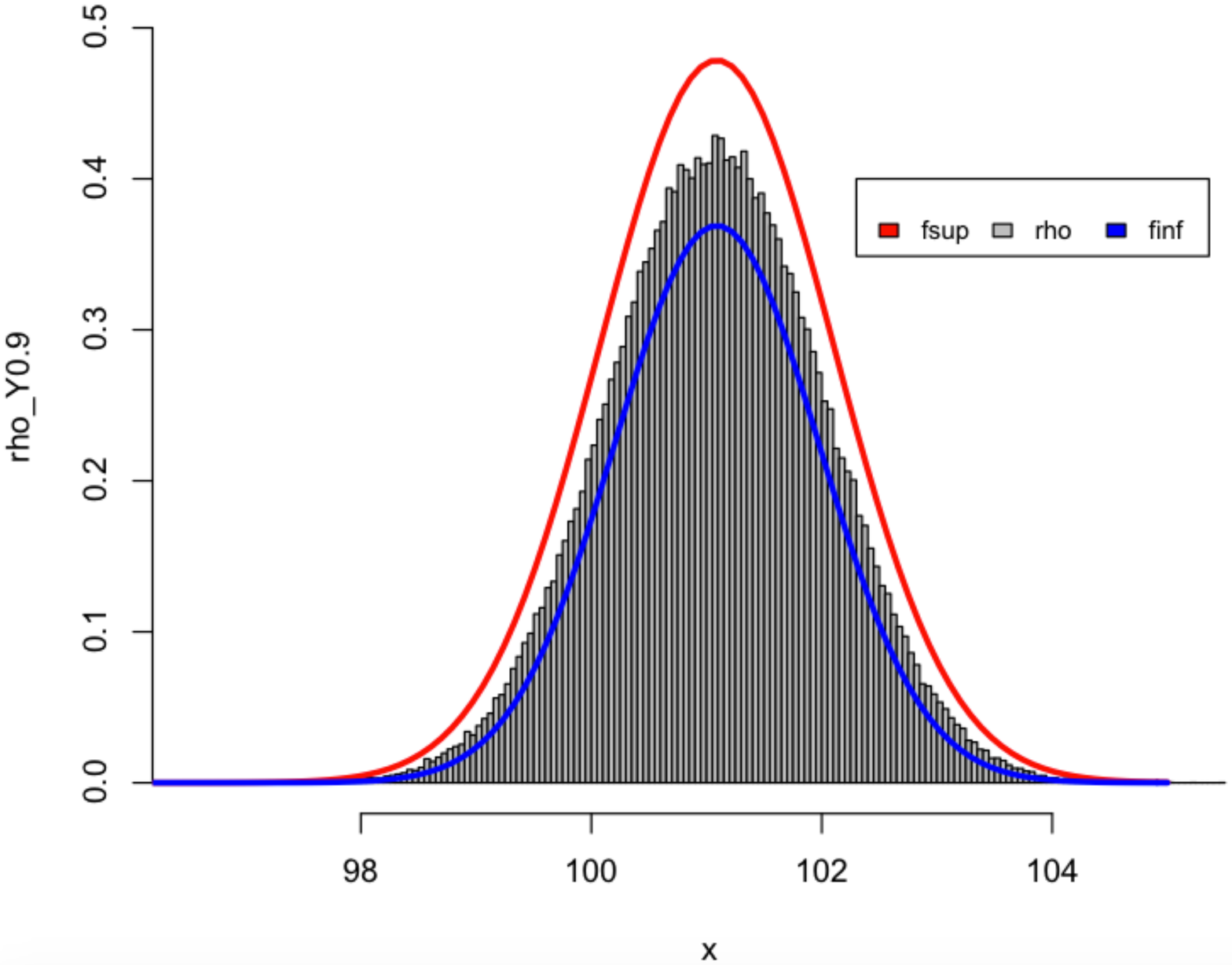}
\includegraphics[scale=0.25]{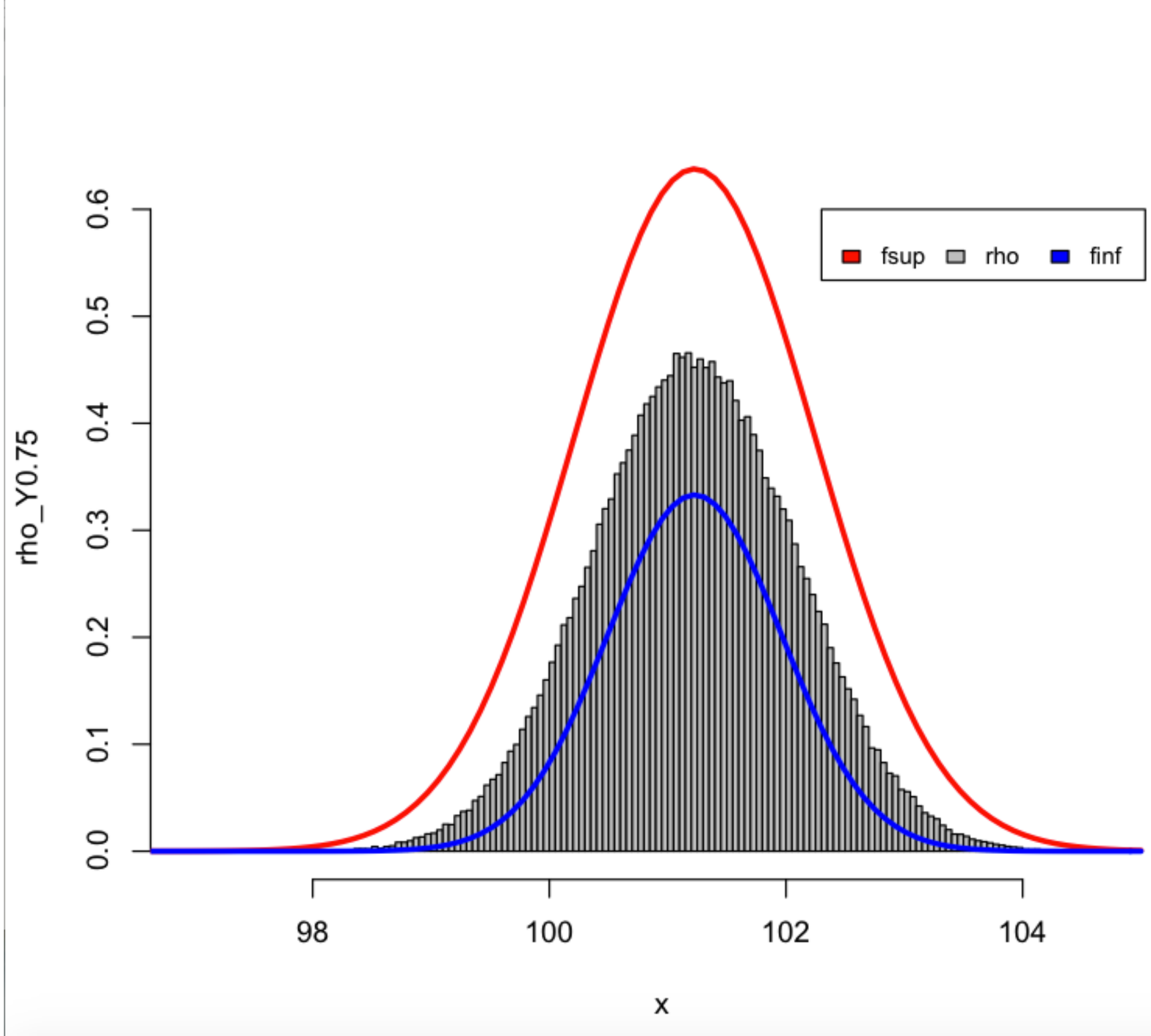}
\includegraphics[scale=0.25]{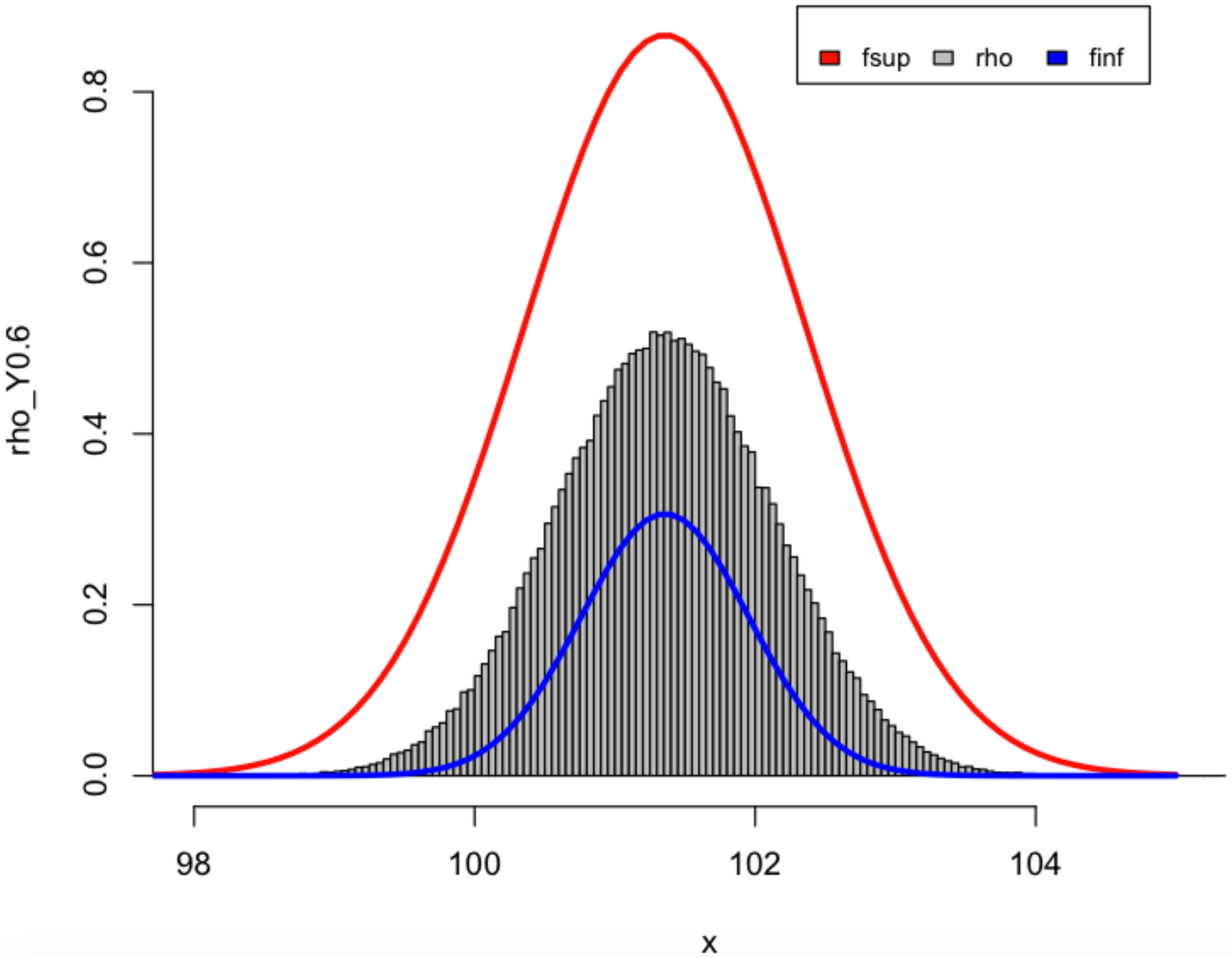}
\includegraphics[scale=0.25]{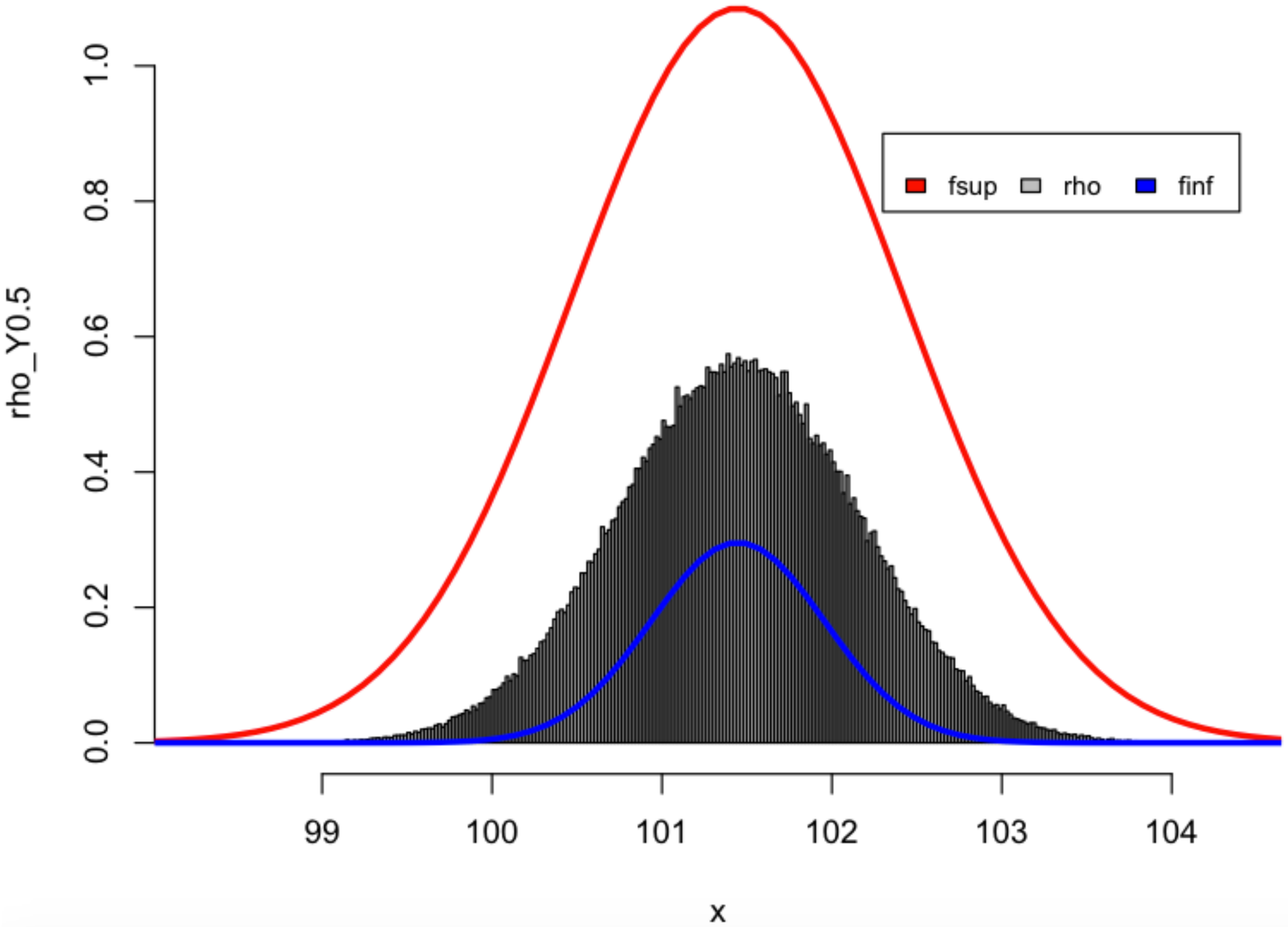}
 \end{center}
\end{figure}

\begin{center}
\begin{tabular}{|c|c|c|c|c|}\hline
Time $t$& 0.9 &0.75& 0.6& 0.5\\
\hline $\mathbb E[Y_t]$& 101.089& 101.228&101.357 &101.446\\
\hline $\text{Var}[Y_t]$&0.89847&0.749607&0.599469&0.503076\\
\hline\end{tabular}
\end{center}
\paragraph{Interpretation} The closer $t$ is to $T$, the better the approximation is using Proposition \ref{estimesra}. Besides, this method guarantees Gaussian tails to control extreme events which is fundamental to validate the model developed in \cite{shamarova_etal.} by comparing the obtained data with those induced by Gillepsie Method (see \cite{shamarova_etal.} and \cite{Gillespie77} for more details).\vspace{0.3em}

\noindent Notice finally that the variance of $Y_t$ seems to be a decreasing function of the time. This is not surprising since $Y_0$ is deterministic.
\subsection{Example 2. An example in the non-Markovian case}
We now propose to extend the model developed by Shamarova, Ramos and Aguiar (see the previous Example 1) to the non-Markovian setting. This extension might be quite relevant when we study the synthesis of protein in some kind of cells (see for instance \cite{BratsunVolfsonHastyTsimring,Leoncini_Thesis,FromionLeonciniRobert}).
Assume that there exist $\alpha\in \mathbb R$, $\beta>0$ and $\gamma\geq 0$ such that $\xi=\alpha+\beta W_T +\gamma\int_0^T W_s ds.$ Hence, BSDE \eqref{bsdegene}becomes:

\begin{equation}\label{bsdesraextension}
Y_t=\alpha+\beta W_T+\gamma\int_0^T W_s ds+\int_t^T \left(R \dfrac{aY_s^2}{1+aY_s^2}-\rho Y_s\right) ds -\int_t^T Z_s dW_s,
\end{equation}
According to Theorem \ref{chapBiolipschitz:existence} and Theorem \ref{Mbio:thm:d12}, BSDE \eqref{bsdesraextension} admits a unique solution $(Y,Z)$ such that for any $t\in [0,T]$, $Y_t\in \mathbb D^{1,2}$ and $Z\in L^2([t,T];\mathbb D^{1,2})$. According to Proposition \ref{estimesra}, for any $t\in [0,T]$, $Y_t$ admits a density with respect to the Lebesgue measure denoted by $\rho_{Y_t}$ such that $\rho_{Y_t}$ has Gaussian estimates, satisfying the following inequalities for any $x\in \mathbb R$ 

\begin{equation}\label{densityestimates}f_i(x)\leq \rho_{Y_t}(x) \leq f_s(x),\end{equation}
where
$$f_i(x)=\frac{C_{Y_t}}{(\beta+\gamma T)^2 t}e^{-2\overline{C_{a,R,\rho}}(T-t)} e^{-e^{-2\underline{C_{a,R,\rho}}(T-t)}\frac{(x-\mathbb E[Y_t])^2}{2\beta^2 t}}, $$
$$f_s(x)= \frac{C_{Y_t}}{\beta^2 t}e^{-2\underline{C_{a,R,\rho}}(T-t)} e^{-e^{-2\overline{C_{a,R,\rho}}(T-t)}\frac{(x-\mathbb E[Y_t])^2}{2(\beta+\gamma T)^2t}}.$$
\section{Applications to the classical pricing problem}\label{Mbio:section:finance}

\subsection{General model and comments}
The problem of pricing in finance using BSDE was
first developed in \cite{ELKPengQuenez}. Consider a financial market in which an agent invests in a riskless asset, denoted by $S^0$, whose the dynamics is given by the short rate of the market, denoted by $r$, and a risky asset, denoted by $S$, whose the dynamic is given through a predictable process, called the risk premium and denoted by $\theta$. Let now $\xi$ be a contingent claim. The classical pricing problem consists in finding an hedging strategy $Z$ and a price $y_0$ such that the terminal wealth of the agent is $\xi$. It was showed in \cite{ELKPengQuenez} that this pricing problem can be reduced to solve the following
stochastic linear BSDE, when $S$ is a geometric Brownian motion
\begin{equation}\label{bsdefinance}
dY_t =(r_tY_t+\theta_tZ_t)dt+Z_tdW_t, \; Y_T=\xi. 
\end{equation}
More generally, we set the following assumption, which enlarge the range of possible applications to this study
\begin{itemize}
\item[$\mathbf{(S)}$]
Let an asset $S$ such that for any $\mathcal F_T$-measurable square integrable random variable $\xi$, the pricing problem can be reduced to study BSDE \eqref{bsdefinance}, where the process $\theta$ depends on the dynamic of $S$.
\end{itemize}
\begin{Remark}
We provide in this remark two classical examples of process $S$ satisfying the previous Assumption $\mathbf{(S)}$.

\vspace{0.5em}
\noindent 
$(\rm{aB})$ Assume that the asset $S$ is an arithmetic Brownian motion, with the following dynamic
\begin{equation*}
dS_t= b_tdt+\sigma_t  dW_t,\; S_0=x\in \mathbb R, 
\end{equation*}
where
$b$ and $\sigma$ are $\mathbb F$-predictable processes with $\sigma_t>0,\; \mathbb P-a.s.$.  Given an $\mathcal F_T$-measurable square integrable random variable $\xi$, using the self-financing Property, one can easily show that the corresponding pricing problem can be reduced to solve BSDE \eqref{bsdefinance} with $\theta:= \frac{b-r}{\sigma}$. In this case, the process $Y$ provides the value of the problem and the process $Z/\sigma$ gives the optimal number of asset owned at time $t$ to solve the pricing problem.

\vspace{0.5em}
\noindent $(\rm{gB})$ Assume that the dynamic of the asset $S$ is given by
\begin{equation*}
dS_t= b_tS_tdt+\sigma_t S_t dW_t,\; S_0=x\in \mathbb R, 
\end{equation*}
where
$b$ and $\sigma$ are $\mathbb F$-predictable processes with $\sigma_t>0,\; \mathbb P-a.s.$ Given an $\mathcal F_T$-measurable square integrable random variable $\xi$, it was showed in \cite{ELKPengQuenez} that the corresponding pricing problem can be reduced to solve BSDE \eqref{bsdefinance} with $\theta:= \frac{b-r}{\sigma}$. In this case, the process $Y$ provides the value of the problem and the process $Z/\sigma$ gives the optimal quantity of money invested in the risky asset to solve the pricing problem.
\end{Remark}

\noindent Most of models assume that $r$ is bounded to simplify the study. However,
as noticed in \cite{ELKHuang}, the assumption on the boundedness of the short rate $r$ rarely holds in
a market. In this section, we investigate the existence of densities for the laws of the components of the solution to \eqref{bsdefinance}. In this model, Assumptions $\mathbf{(A_1)}$ and $\mathbf{(A_2)}$ $(i)$ above in Section \ref{sectionaffinestoBSDE} become

\vspace{0.5em}
\noindent \textbf{(H1)} For any $p>0$, $\mathbb E\left[ e^{p\int_0^T r_s ds}\right]<+\infty$ and $\left(\int_0^\cdot \theta_t dW_t\right)$ is a BMO-martingale.

\noindent We thus have the following lemma
\begin{Lemma}\label{lemma:signY}Assume that $(\mathbf{H1})$ holds and that for any $p>0$, $\mathbb E\left[ |\xi|^p\right]<+\infty$. Then, BSDE \eqref{bsdefinance} admits a unique solution $(Y,Z)\in\mathcal S_{p}\times \mathcal H_{p}$ for any $p>1$. Besides, if $\xi \geq 0, \; \mathbb P-a.s.$ (resp. $\xi \leq 0, \; \mathbb P-a.s.$), then for any $t\in [0,T]$, $Y_t\geq 0,\mathbb P-a.s.$ (resp. $\xi \leq 0, \; \mathbb P-a.s.$) 
\end{Lemma}
\begin{proof} The proof of the existence of a unique solution $(Y,Z)$ in $\in\mathcal S_{p}\times \mathcal H_{p}$ is a consequence of Theorem \ref{thm:affine:existence}. Using a linearisation, we get
$$Y_t=\mathbb E^{\mathbb Q}_t\left[\xi e^{-\int_t^T r_s ds} \right], $$
where $$\frac{d\mathbb Q}{d\mathbb P}:= e^{-\int_0^T \theta_s dW_s -\frac12 \int_0^T |\theta_s|^2 ds}.$$
Thus, we notice that the sign of the $Y$ process is given by the sign of $\xi$.
\end{proof}

\subsection{Application to Va\v{s}\`i\v{c}ek Model}
Let $a,b\geq 0$ and $\varpi>0$. Assume that the rate of the market $r$ is the solution of the following SDE.
\begin{equation}\label{vasicek}dr_t=a(b-r_t)dt+\varpi dW_t,\; r_0\in\mathbb R.\end{equation}

\begin{Lemma}\label{lemma:derivepos} Let $r:=(r_t)_{t\in [0,T]}$ be the solution to SDE \eqref{vasicek}. Then, for any $p>1$, $q\geq 1$ and for any $t\in [0,T]$, $r_t\in \mathbb D^{q,p}$. Besides, for any $0\leq u\leq t\ $, $D_u r_t=\varpi \geq 0$, $\mathbb P-a.s.$ and for any $q>1$, $D^q r_t =0, \; \mathbb P-a.s.$.
\end{Lemma}
\begin{proof}
Let $r:=(r_t)_{t\in [0,T]}$ be the solution to SDE \eqref{vasicek}. Notice that $r_t$ is Malliavin differentiable (see e.g. \cite[Theorem 2.2.1]{Nualart_Book}). Besides, as an Ornstein-Uhlenbeck process, $r_t$ can be computed explicitly
\begin{equation}\label{vasicek:expression}
r_t= r_0 e^{-at}+b(1-e^{-at})+\varpi e^{-at}\int_0^t e^{as}dW_s.
\end{equation}
Taking the Malliavin derivative, one obtains directly that for any $r_t \in \mathbb D^{q,p}$ for any $p>1$, $q\geq 1$. Besides for any $0\leq u\leq t \leq T$, $D_u r_t =\varpi \geq 0$, $\mathbb P-a.s.$ and for any $q>1$, $D^q r_t =0, \; \mathbb P-a.s.$
\end{proof}

\noindent Since we aim at applying Bouleau and Hirsch Criterion (see Theorem \ref{BH}), we first show that the components $Y_t$ and $Z_t$ of the solution to BSDE \eqref{bsdefinance} are Malliavin differentiable. In this section we will work under Assumption $\mathbf{(\Theta)}$ (see Section \ref{soussection:z:linear}) since we aim at applying the results of Section \ref{soussection:z:linear} to investigate the existence of densities for both the $Y_t$ and the $Z_t$ components. Although this assumption is really restrictive, we can not do better as explained in Remark \ref{rem:pb:signZ}. However, for the following result dealing with the Malliavin differentiability of $Y_t$ and $Z_t$, one could make weaker Assumption $\mathbf{(\Theta)}$ by considering that Assumption $\mathbf{(A_1)}$ holds.
\begin{Proposition}\label{bsde_MD_vasicek} Let $\xi \in \mathbb D^{1,p}$ for any $p>1$. Let $r$ be the unique solution to SDE \eqref{vasicek} and $\theta$ satisfying Assumption $\mathbf{(\Theta)}$. Then, BSDE \eqref{bsdefinance:asian} admits a unique solution $(Y,Z)\in\mathcal S_{p}\times \mathcal H_{p}$ for any $p>1$. Besides, for any $p>1$ and $t\in [0,T]$, $Y_t\in \mathbb D^{1,p}$ and $Z\in L^2([t,T]; \mathbb D^{1,p}) $.
\end{Proposition}
\begin{proof}
By noticing that Assumptions $\mathbf{(A_1)}$, $\mathbf{(A_2)}$, $\mathbf{(DA_1)}$ and $\mathbf{(DA_2)}$ hold and by applying Theorem \ref{chapBiothm:diff:affine:BC}, we deduce that BSDE \eqref{bsdefinance:asian} admits a unique solution $(Y,Z)\in\mathcal S_{p}\times \mathcal H_{p}$ for any $p>1$ and that if $t\in [0,T]$, $Y_t\in \mathbb D^{1,p}$ and $Z\in L^2([t,T]; \mathbb D^{1,p}) $ for any $p>1$.
\end{proof}

\noindent In this particular model and as said in Section \ref{chapBioMbio:section:discussion}, we provide now conditions on $\xi$ and its Malliavin derivatives ensuring existence of densities for both the law of the $Y_t$ component and for the law of the $Z_t$ component of the solution to BSDE \eqref{bsdefinance}.
\begin{Theorem} \label{thm:finance:densityYZ}Let $\xi \in \mathbb D^{1,p}$ for any $p>1$. Assume that $\mathbf{(\Theta)}$ holds and that one of the following two assumptions is satisfied for $A\subset \Omega$ such that $\mathbb P(A)>0$

\vspace{0.5em}
\noindent $\boldsymbol{(\xi+)}$ $\xi\geq 0$, $D_u \xi \leq  0,\ \lambda(du)-a.e.,\;   D_u \xi < 0,\ \lambda(du)-a.e.\text{ on } A,$

\vspace{0.5em}
\noindent $\boldsymbol{(\xi-)}$ $\xi\leq 0$, $D_u \xi \geq  0,\ \lambda(du)-a.e.,\;   D_u \xi > 0,\ \lambda(du)-a.e.\text{ on } A,$

\noindent then for any $t\in (0,T]$, the law of $Y_t$ is absolutely continuous with respect to the Lebesgue measure.\vspace{0.3em}

\noindent Assume now that $\xi \in \mathbb D^{2,p}$ for any $p>1$ and assume in addition to $\boldsymbol{(\xi+)}$ that  
\begin{equation}\label{d2xi_densityZ}D_v(D_u \xi) \geq  0,\ (\lambda\otimes \lambda)(du,dv)-a.e.,\;   D_v(D_u \xi) >0,\ (\lambda\otimes \lambda)(du,dv)-a.e.\text{ on } A\end{equation}
or in addition to $\boldsymbol{(\xi-)}$ that
 \begin{equation}\label{d2xi_densityZ:-}
 D_v(D_u \xi) \leq  0,\ (\lambda\otimes \lambda)(du,dv)-a.e.,\;   D_v(D_u \xi) <0,\ (\lambda\otimes \lambda)(du,dv)-a.e.\text{ on } A, \end{equation}
 then the law of $Z_t$ has a density with respect to the Lebesgue measure.
\end{Theorem}
\begin{proof}
We denote by $(Y,Z)$ the unique solution in $\mathcal S_{p}\times \mathcal H_{p}$ for any $p>1$ of BSDE \eqref{bsdefinance} with for any $p>1$ and $t\in [0,T]$, $Y_t\in \mathbb D^{1,p}$ and $Z\in L^2([t,T]; \mathbb D^{1,p}) $ by using Proposition \ref{bsde_MD_vasicek}. Let Assumption $\boldsymbol{(\xi+)}$ be true. Then, according to Theorem \ref{chapBiothm:densityaffine} together with Lemmas \ref{lemma:signY} and \ref{lemma:derivepos}, for any $t\in (0,T]$ the law of $Y_t$ has a density with respect to the Lebesgue measure. Recall that under Assumption $\boldsymbol{(\xi+)}$, according to Remark \ref{chapBioremarksigne}, we have for any $t\in (0,T]$, $D_uY_t\leq 0, \; \lambda(du)\otimes \mathbb P-a.e.$ By assuming moreover that Conditions \eqref{d2xi_densityZ} holds and by applying Theorem \ref{densityZ:linear} with $\tilde f(t,Y_t):=-r_t Y_t$, we deduce that $\mathbf{(DZ+)}, \mathbf{(DY-)}$ and $(f-)$ hold. Then $D_v Z_t>0$ for any $0\leq v<t\leq T$, $\mathbb P$-almost surely. Thus, according to Theorem \ref{BH}, the law of $Z_t$ has a density with respect to Lebesgue measure for any $t\in (0,T]$.\vspace{0.3em}

\noindent The proof under $\boldsymbol{(\xi-)}$ is similar as a consequence of Theorem  \ref{densityZ:linear} by showing that Assumptions $\mathbf{(DZ-)}, \mathbf{(DY+)}$ and $(f-)$ hold.
\end{proof}

\subsubsection{Example 1: Asian options}
In this section, we investigate pricing problems of Asian options, \textit{i.e.} where the liability $\xi$ is a function of the mean of the risky asset $S$. We assume that Assumption $(\mathbf{S})$ holds, thus the pricing problem is reduced to solve the affine non-Markovian BSDE 
\begin{equation}\label{bsdefinance:asian}Y_t= \xi-\int_t^T(r_s Y_s+\theta_s Z_s )ds -\int_t^TZ_s dW_s,\; \xi=f\left(\int_0^T g(W_s)ds \right),\end{equation}
where $f,g$ are two continuous maps from $\mathbb R$ into $\mathbb R$.

\begin{Proposition}\label{propvasicek_density} Assume that $\mathbf{(\Theta)}$ hold. Let $r$ be the unique solution to SDE \eqref{vasicek}. Assume moreover that $f,g$ are twice differentiable $\lambda(dx)$-a.e. and one of the following assumption is satisfied

\vspace{0.5em}
\noindent $\mathbf{(\mathcal A1+)} $
\begin{itemize}
\item[$(i)$] $f\geq 0$, $f'\geq 0$, $g'\leq 0$ and $f'>0,\; g'<0$ on a set A with positive Lebesgue measure,\vspace{0.3em}
\item[$(ii)$] moreover $f"\geq 0$, $g''\geq 0$, and $f''>0$ or $g''>0$ on A,
\end{itemize}\vspace{0.5em}

\noindent $\mathbf{(\mathcal A2+)} $
\begin{itemize}
\item[$(i)$] $f\geq 0$, $f'\leq 0$, $g'\geq 0$ and $f'<0,\; g'>0$ on a set A with positive Lebesgue measure,\vspace{0.3em}
\item[$(ii)$] moreover $f"\geq 0$, $g''\leq 0$, and $f''>0$ or $g''<0$ on A,
\end{itemize}\vspace{0.5em}

\noindent $\mathbf{(\mathcal A1-)} $
\begin{itemize}
\item[$(i)$] $f\leq 0$, $f'\geq 0$, $g'\geq 0$ and $f'>0,\; g'>0$ on a set A with positive Lebesgue measure,\vspace{0.3em}
\item[$(ii)$] moreover $f"\leq 0$, $g''\leq 0$, and $f''<0$ or $g''<0$ on A.
\end{itemize}\vspace{0.5em}

\noindent $\mathbf{(\mathcal A2-)} $
\begin{itemize}
\item[$(i)$] $f\leq 0$, $f'\leq 0$, $g'\leq 0$ and $f'<0,\; g'<0$ on a set A with positive Lebesgue \vspace{0.3em}
\item[$(ii)$] moreover $f"\leq 0$, $g''\geq 0$, and $f''<0$ or $g''>0$ on A.
\end{itemize}

\noindent  Then, by denoting $(Y,Z)$ the unique solution of BSDE \eqref{bsdefinance:asian}, for any $t\in (0,T]$ both the law of $Y_t$ and the law of $Z_t$ are absolutely continuous with respect to Lebesgue measure.
\end{Proposition}

\begin{proof} 
Notice that for any $0\leq u\leq T$ we have
$$D_u \xi =f'\left( \int_0^T g(W_s)ds\right)\int_u^T g'(W_s)ds, $$
and for any $0\leq v\leq T$ we have
\begin{align*}
D_v (D_u \xi)&=f''\left( \int_0^T g(W_s)ds\right)\int_u^T g'(W_s)ds\int_v^T g'(W_s)ds\\
&\hspace{1em}+f'\left( \int_0^T g(W_s)ds\right)\int_{u\wedge v}^T g''(W_s) ds. \end{align*}
Thus, by noticing that Assumption $\mathbf{(\mathcal A1+)}$ $(i)$ or $\mathbf{(\mathcal A2+)}$ $(i)$ (resp. Assumption $\mathbf{(\mathcal A1-)}$ $(i)$ or $\mathbf{(\mathcal A2-)}$ $(i)$) ensure that $\boldsymbol{( \xi+)} $ (resp. $\boldsymbol{( \xi-)} $) is true, we deduce from the first part of Theorem \ref{thm:finance:densityYZ} above that the law of $Y_t$ has a density with respect to Lebesgue's measure for any $t\in (0,T]$. Moreover, if Assumption $\mathbf{(\mathcal A1+)}$ $(ii)$ or $\mathbf{(\mathcal A2+)}$ $(ii)$  (resp. Assumption $\mathbf{(\mathcal A1-)}$ $(ii)$ or $\mathbf{(\mathcal A2-)}$ $(ii)$) holds, then Condition \eqref{d2xi_densityZ} is satisfied (resp. \eqref{d2xi_densityZ:-}). By applying Theorem \ref{thm:finance:densityYZ} we deduce that $Y_t$ and $Z_t$ have absolutely continuous law with respect to Lebesgue measure.
 \end{proof}
 
\subsubsection{Example 2: Lookback options}
In this section, we aim at applying Theorem \ref{thm:finance:densityYZ} to lookback options. Let Assumption $\mathbf{(S)}$ be true.
Set $M:= (M_t)_{t\in [0,T]}$, where $M_t= \sup_{s\in [0,t]}W_s$. 
The following lemma is a direct consequence of \cite[Lemma 1.1]{GKH}, \cite[Remark 8.16 and Problem 8.17]{KS}.
\begin{Lemma}\label{lemmaGKH}
 $M_t \in \mathbb D^{1,2}$ and
$D_r M_t= \mathbf 1_{r\leq \tau_t}$, where $\tau_t$ is almost surely unique such that defined by $W_{\tau_t}= M_t. $ More precisely, for any $0\leq s\leq t$,
$$\mathbb P (\tau_t \leq s)= \frac 2\pi \text{arcsin}\sqrt{\frac st} .$$
\end{Lemma}

\noindent We consider the affine non-Markovian BSDE
\begin{equation}\label{bsdefinanceexemple}Y_t= \xi- \int_t^T(r_s Y_s+\theta_s Z_s )ds -\int_t^TZ_s dW_s, \; \xi=f(M_T),\end{equation} where $f$ is a continuous mapping from $\mathbb R$ into $\mathbb R$. We have the following proposition which is a consequence of Lemma \ref{lemmaGKH} together with Theorem  \ref{chapBiothm:densityaffine}.
\begin{Proposition} Let $Y$ be the first component of the solution of BSDE \eqref{bsdefinanceexemple}, hence for any $t\in (0,T]$, if $f$ is differentiable $\lambda(dx)$-a.e. and one of the following two assumptions is satisfied
\begin{itemize}
\item[$(lb+)$] $f\geq 0$ and $f'\leq 0$ and $f'<0$ on a set with positive Lebesgue measure,
\item[$(lb-)$] $f\leq 0$ and $f'\geq 0$ and $f'>0$ on a set with positive Lebesgue measure,
\end{itemize}
then the law of $Y_t$ is absolutely continuous with respect to the  Lebesgue measure.
\end{Proposition}

\begin{Remark}\label{remark_lookback}
Since $\xi:= M_T$ is not twice Malliavin differentiable (see \cite{GKH}), \textit{i.e.} $\xi$ does
not belong to $\mathbb D^{2,p}$ whatever $p\geq 1$, we cannot reproduce the proof of Proposition \ref{propvasicek_density} to study
the problem of existence of density for $Z_t$.
\end{Remark}
\section*{Acknowledgments}
The author thanks Dylan Possama\"i and Anthony R\'eveillac for conversations and precious advice in the writing of this paper. The author is grateful to R\'egion Ile-De-France for financial support. 

\end{document}